\documentclass[11pt,leqno,twoside]{article}
\def\cvd{\hfill$\Box$}

\def\corr{\longleftrightarrow}

\def\int{\mathbb{Z}}
\def\C{\mathbb{C}}
\def\R{\mathbb{R}}

\def\Z{\mathbb{Z}}
\def\N{\mathbb{N}}
\def\e{\varepsilon}
\def\Ue{{\cal U}_{\varepsilon}({\mathfrak g})}

\def\Aut{\text{Aut}~}
\def\ch{{\rm{char}~}}

\def\sph{\text{sph}}
\def\O{{\cal O}}
\def\cu{{\cal CU}}

\def\h{{\mathfrak h}}

\def\K{{k}}

\def\proof{\noindent{\bf Proof. }}
\def\Pf{\proof}

\def\pf{\proof}

\def\rk{{\rm rk}}
\def\ov{\overline}

\def\a{\alpha}
\def\f{\varphi}
\def\b{\beta}
\def\d{\delta}
\def\l{\lambda}
\def\g{\gamma}
\def\s{\sigma}
\def\r{\rho}

\def\t{\tau}

\def\<#1{\langle #1\rangle}
\def\o#1{|\ \! #1\!\ |}

\def\wJ{w_{\!_J}}

\def\vuoto{\varnothing}

\usepackage{times}
\usepackage{graphics}
\usepackage{amssymb}
\usepackage{amscd}
\usepackage{amsmath}
\usepackage{fancyhdr}

\input xy
\xyoption{all}

\title{A classification of unipotent spherical conjugacy classes\\ in bad characteristic}

\newtheorem{theorem}{Theorem}[section]
\newtheorem{lemma}[theorem]{Lemma}
\newtheorem{corollary}[theorem]{Corollary}
\newtheorem{proposition}[theorem]{Proposition}
\newtheorem{definition}[theorem]{Definition}
\newtheorem{remark}[theorem]{Remark}

\newcounter{tigre}
\def\totable{\refstepcounter{tigre}\arabic{tigre}}

\fancypagestyle{plain}{%
\fancyhf{}

}

\author{Mauro Costantini\\
Dipartimento di Matematica Pura ed Applicata\\
Torre Archimede - via Trieste 63 - 35121 Padova - Italy\\
email: costantini@math.unipd.it }
\date{}
\begin{document}
\baselineskip=20pt
\maketitle
\begin{abstract}
  Let $G$ be a simple algebraic group over an
  algebraically closed field $k$ of bad characteristic. We classify the spherical unipotent conjugacy classes of $G$. We also show that if the characteristic of $k$ is 2, then the fixed point subgroup of every involutorial automorphism (involution) of $G$ is a spherical subgroup of $G$.
   \end{abstract}
\section{Introduction}

\newcounter{equat}
\def\theequat{(\arabic{equat})}
\def\equat{\refstepcounter{equat}$$~}
\def\endequat{\leqno{\boldsymbol{(\arabic{equat})}}~$$}

\newcommand{\elem}[1]{\stackrel{#1}{\longto}}
\newcommand{\map}[1]{\stackrel{#1}{\to}}
\def\imp{\Rightarrow}
\def\Imp{\Longrightarrow}
\def\iff{\Leftrightarrow}
\def\Iff{\Longleftrightarrow}
\def\to{\rightarrow}
\def\longto{\longrightarrow}
\def\injto{\hookrightarrow}
\def\rtordu{\rightsquigarrow}

Let $G$ be a simple algebraic group over an algebraically closed field $\K$, with Lie algebra $\mathfrak g$. In this paper we determine the unipotent spherical conjugacy classes of $G$ (we recall that a conjugacy class $\O$ in $G$ is
called {\it spherical} if a Borel subgroup of $G$ has a
dense orbit on $\O$) when the characteristic of $k$ is bad for $G$. There has been a lot of work related to this field in the cases of good characteristic.

To fix the notation, $B$ is a Borel subgroup of $G$, $T$ a maximal torus of $B$, $B^-$ the Borel subgroup opposite to $B$, $\{\alpha_1,\ldots,\alpha_n\}$ the set of simple roots with respect to the choice of $(T,B)$.
Let $W$ be the Weyl group of $G$ and let us denote by
$s_i$ the reflection corresponding to the simple root $\alpha_i$: 
$\ell(w)$ is the length of the element $w\in W$ and $\rk(1-w)$ is the rank of $1-w$ in the geometric representation of $W$. 

Initially, spherical $G$-orbits have been studied in the context of Lie algebras (\cite{pany}, \cite{pany2}) in characteristic zero. The classification of spherical nilpotent orbits has been obtained by Panyushev: in terms of height, a nilpotent orbit $\O\subset \mathfrak g$ is spherical if and only if its height is at most 3, which means 2 or 3 if $\O$ is not the zero orbit.  Equivalently, $\O$ is spherical if and only if it contains an element of the form $e_{\gamma_1}+\cdots +e_{\gamma_t}$, where $\gamma_1,\ldots,\gamma_t$ are pairwise orthogonal simple roots
(Panyushev \cite{pany}, \cite{pany5}).

More recently, in \cite{CCC}, we put our attention to spherical conjugacy classes in $G$ over $\C$. We classify all spherical conjugacy classes by means of the Bruhat decomposition:
a conjugacy class $\O\subset G$ is spherical if and only if $\dim{\cal O}=\ell(w)+rk(1-w)$, where $w=w(\O)$ is the unique element of $W$ such that $\O\cap BwB$ is open dense in $\O$ (we observe that the classification given in \cite{CCC} over the complex numbers, holds in general for characteristic zero).

In \cite{FR} the authors classify spherical nilpotent orbits in good characteristic using Kempf-Russeau theory: the classification is the same as in the case of zero characteristic. In \cite{Giovanna-good}, the author obtains the classification of all spherical conjugacy classes in good, odd or zero characteristic by means of the Bruhat decomposition and by exploiting another characterization of spherical conjugacy classes available in good odd or zero characteristic, namely a conjugacy class $\O$ is spherical if and only if $\{y\in W\mid \O\cap ByB\not=\emptyset\}\subseteq\{w\in W\mid w^2=1\}$ (\cite{Car}, Theorem 2.7, \cite{Car2}, Theorem 5.7).

In the present paper we complete the picture for unipotent spherical conjugacy classes by considering bad characteristic. Our strategy is to exhibit for each group $G$ a set $\O(G)$ of unipotent conjugacy classes, show that each element in $\O(G)$ is spherical, and finally show that each conjugacy class not in $\O(G)$ is not spherical. It turns out that in bad characteristic the classification of spherical unipotent conjugacy classes is the ``same'' as in zero characteristic, unless $p=2$ in type $C_n$ and $F_4$, or $p=3$ in type $G_2$. In these cases there are more classes than in characteristic zero. In particular, if $p=2$ then the conjugacy class of a non-trivial unipotent element $u$ is spherical if and only if $u$ is an involution.

It is well known that in zero or odd characteristic the fixed point subgroup $H$ of any involutory automorphism $\s$ of $G$ is a spherical subgroup (i.e. $G/H$ is a spherical homogeneous space). This was proved by Vust in \cite{Vust} in characteristic zero (see also \cite{Matsuki} over $\C$). Then Springer extended the result to odd characteristic in \cite{Springerinv}. In \cite{Seitz}, Seitz gives an alternative proof of Springer's result. Here we prove that the result also holds in characteristic 2.

The paper is structured as follows. In Section \ref{bepi} we introduce the notation. In Section \ref{classificazione} we recall some basic
facts about the classification of unipotent conjugacy classes in bad characteristic and determine the spherical ones. We also give the list of all spherical unipotent conjugacy classes $\O$ for which there is an element $u$ in $\O$ of the form $u= x_{\gamma_1}(1)\cdots x_{\gamma_t}(1)$, where $\gamma_1,\ldots,\gamma_t$ are pairwise orthogonal simple roots.

In Section \ref{simmetrici} we prove that if $G$ is a reductive connected algebraic group in characteristic 2, and $\s$ is any involutory automorphism of $G$, then  the fixed point subgroup $H$ of $\s$ is a spherical subgroup of $G$.

\setcounter{equation}{0}
\section{Preliminaries}\label{bepi}

We denote by $\C$ the complex numbers, by $\R$ the reals, by
$\Z$ the integers and by $\N$ the natural numbers. 

Let $A = (a_{ij})$ be a finite indecomposable
Cartan matrix of rank $n$ with associated root system $\Phi$, and let $k$ be an algebraically closed field of characteristic  $\ch k=p$. Let $G$ be a simple algebraic group over $\K$ associated to $A$, with
 Lie 
algebra
${\mathfrak g}$. We fix a maximal torus $T$ of $G$, and a Borel subgroup $B$
 containing $T$:  $B^-$ is the Borel subgroup opposite to $B$, $U$ (respectively $U^-$) is the unipotent radical of $B$ (respectively of $B^-$).  We denote by $\h$ the Lie
algebra of $T$. Then  $\Phi$ is the set of roots relative
to $T$, and $B$ determines the set of positive roots  
$\Phi^+$, and the simple roots $\Delta=\{\alpha_1,\ldots,\alpha_n\}$.  We fix a total ordering on $\Phi^+$ compatible with the height function.
We shall use the numbering and the description of the simple roots in terms of the canonical basis $(e_1,\ldots, e_k)$ of an appropriate $\R^k$ as in \cite{bourbaki}, Planches I-IX. For the exceptional groups, we shall write $\b=(m_1,\ldots,m_n)$ for $\b=m_1\a_1+\ldots+m_n\a_n$.
We denote by $P$ the weight lattice, by $P^+$ the monoid of dominant weights and by $W$ the Weyl
group; $s_i$ is the simple reflection associated to $\alpha_i$,
$\{\omega_1,\ldots,\omega_n\}$ are the fundamental weights, $w_0$ is the longest element of $W$. 
The real space $E=\R P$  is a Euclidean space, endowed with the scalar product
$(\a_i,\alpha_j) = d_ia_{ij}$. Here
$\{d_1,\ldots,d_n\}$  are
relatively prime positive integers such that if $D$ is the
diagonal matrix with entries $d_1,\ldots,d_n$, then $DA$ is
symmetric. 

 We put
$\Pi=\{1,\ldots,n\}$ and we fix a Chevalley basis $\{h_{i},  i\in \Pi; e_\a,\a\in\Phi\}$
 of $\frak g$.

We use the notation $x_\a(\xi)$, $h_\a(z)$, for $\a\in \Phi$, $\xi\in \K$, $z\in\K^\ast$ as in \cite{yale}, \cite{Carter1}. For $\a\in \Phi$ we put $X_\a=\{x_\a(\xi)\mid \xi\in \K\}$, the root-subgroup corresponding to $\a$, and $H_\a=\{h_\a(z)\mid z\in \K^\ast\}$. 
  We identify $W$ with $N/T$, where $N$ is the normalizer of $T$:  given an element $w\in W$ we shall denote
a representative of $w$ in $N$ by $\dot{w}$. We choose the $x_\a$'s  so that, for all $\a\in \Phi$, $n_\a=x_\a(1)x_{-\a}(-1)x_\a(1)$
lies in $N$ and has image the reflection $s_\a$ in $W$. Then 
\begin{equation}
x_\a(\xi)x_{-\a}(-\xi^{-1})x_\a(\xi)=h_{\a}(\xi)n_\a\quad,\quad
n_\a^2=h_{\a}(-1)
\label{relazioni}
\end{equation}
for every $\xi\in\K^\ast$, $\a\in\Phi$ (\cite{springer}, Proposition 11.2.1).


For algebraic groups we use the notation in \cite{Hu2}, \cite{Carter2}. In particular, 
for $J\subseteq \Pi$, $\Delta_J=\{\a_j\mid j\in J\}$, $\Phi_J$ is the corresponding root system, $W_J$ the Weyl group, $P_J$ the standard parabolic subgroup of $G$, $L_J=T\<{X_\a\mid \a\in \Phi_J}$ the standard Levi subgroup of $P_J$.
For $z\in W$ we put 
$
U_z=U\cap z^{-1}U^-z$.
Then the unipotent radical $R_uP_J$ of $P_J$ is $U_{w_0\wJ}$,
where $\wJ$ is the longest element of $W_J$.
Moreover
$
U\cap L_J=U_{\wJ}$
is a maximal unipotent subgroup of $L_J$.

If $\Psi$ is a subsystem of type $X_r$ of $\Phi$ and $H$ is the subgroup generated by $X_\a$, $\a\in \Psi$, we say that $H$ is a $X_r$-subgroup of $G$.

If $X$ is  $G$-variety and $x\in X$, we denote by $G.x$ the $G$-orbit of $x$ and by  $G_x$ the isotropy subgroup of $x$ in $G$. If the homogeneous space $G/H$ is spherical, we say that $H$ is a spherical subgroup of $G$.

If $x$ is an element of a group $K$ and $H\leq K$, we shall also denote by $C(x)$ the centralizer of $x$ in $K$, and by $C_H(x)$ the centralizer of $x$ in $H$. If $x$, $y\in K$, then $x\sim y$ means that $x$, $y$ are conjugate in $K$.
For unipotent classes in exceptional groups we use the notation in \cite{Carter2}. We use the description of centralizers of involutions as in \cite{Iwa}, \cite{AS}.

We denote by $\Z_r$ the cyclic group of order $r$.

For each conjugacy class $\O$ in $G$, $w=w(\O)$ is the unique element of $W$ such that $BwB\cap \O$ is open dense in $\O$.

\setcounter{equation}{0}
\section{The classification}\label{classificazione}

We recall that
the bad primes for the individual types of simple groups are as follows:

none when $G$ has type $A_n$;

$p=2$ when $G$ has type $B_n$, $C_n$, $D_n$;

$p=2$ or $3$ when $G$ has type $G_2$, $F_4$, $E_6$, $E_7$;

$p=2$, $3$ or $5$ when $G$ has type $E_8$.

\noindent
One may find a detailed account of the classification of both unipoten classes and nilpotent orbits in bad characteristic in \cite{Carter2}, \S 5.11.

\medskip

To deal with the classical groups with $p=2$, we recall that the unipotent classes were determined by Wall in \cite{Wall} (the nilpotent orbits were determined by Hesselink in \cite{Hesselink}).  For convenience of the reader, here we recall the classification of unipotent classes in the classical groups following \cite{spal}, \S 2.
Suppose $G=GL(n)$ (any characteristic), $u$ a unipotent element of $G$. Then one can associate to $u$ a partition $\l=(\l_1,\l_2,\ldots)=1^{c(1)}\oplus 2^{c(2)}\oplus\cdots$ of $n$ with $\l_1\geq\l_2\geq \cdots$, where $c(i)$ is the number of Jordan blocks of $u$ of dimension $i$, for every $i\geq 1$.
In this way the set $\cu(G)$ of unipotent conjugacy classes of $G$
is parametrized by the set of partitions of $n$. We denote by $C_\l$ the unipotent class corresponding to the partition $\l$. The set $\cu(G)$ has a natural partial order: $\O_1\leq \O\iff \O_1\subseteq \ov\O$.
If we partially order the set of partitions of $n$ by $\l\leq \mu\iff\sum_{j=1}^i\l_j\leq \sum_{j=1}^i\mu_j$ for every $i\geq 1$, then the map $\l\to C_\l$ is an isomorphism of p.o.-sets.

Now assume $p=2$.
In this case there exists a homomorphism (central isogeny) of $SO(2n+1)$ onto $Sp(2n)$ which is an isomorphism of abstract groups. We shall therefore deal only with $Sp(2n)$ and $SO(2n)$. Let $\omega$ be an object distinct from $0$ and $1$, and consider the set $\{\omega,0,1\}$ totally ordered by $\omega<0<1$. Assume $G=Sp(2n)\leq GL(2n)$ (resp. $G=O(2n)\leq Sp(2n)$). The unipotent conjugacy classes of $G$ are parametrized by pairs $(\l,\e)$ such that

\noindent
\begin{itemize}
\item[a)]  $\l=1^{c(1)}\oplus2^{c(2)}\oplus\cdots$ is a partition of $2n$ with $c(i)$ even for every odd $i$.
\item[b)] $\e:\N\to\{\omega,0,1\}$ is such that
\begin{itemize}
\item[b$_1$)] $\e(i)=\omega$ if $i$ is odd or $i\geq 1$ and $c(i)=0$.
\item[b$_2$)] $\e(i)=1$ if $i$ is even  and $c(i)$ is odd ($i\not=0$).
\item[b$_3$)] $\e(i)\not=\omega$ if $i$ is even  and $c(i)\not=0$  ($i\not=0$).
\item[b$_4$)] $\e(0)=1$ (resp. $\e(0)=0$).
\end{itemize}
\end{itemize}

The correspondence is obtained as follows. Let $u$ be a unipotent element of $G$. Then $u$ determines a class in $GL(2n)$, hence the partition $\l$ of $2n$. This partition satisfies a). Now, if $i$ is even, $i\not=0$ and $c(i)\not=0$, we put $\e(i)=0$ if $f((u-1)^{i-1}(x),x)=0$ for every $x\in \ker(u-1)^i$, and $\e(i)=1$ otherwise (here $f$ is the bilinear form used to define $Sp(2n)$). In view of condition b), this defines uniquely $\e$. 

We denote by $C_{\l,\e}$ the unipotent class of $G$ corresponding to $(\l,\e)$. We observe that every unipotent class of $Sp(2n)$ intersects $O(2n)$ in a unique class of $O(2n)$. Moreover, the unipotent classes of $O(2n)$ contained in $SO(2n)$  (the connected component of $O(2n)$)  are those for which $\l_1^\ast$ is even ($\l^\ast$ is the dual partition of $\l$).  If all $\l_i$'s and $c(i)$'s are even and if $\e(i)\not=1$ for every $i$, then the unipotent class $C_{\l,\e}$ of $O(2n)$ splits into two classes of $SO(2n)$. All the other unipotent classes in $SO(2n)$ are unipotent classes in $O(2n)$. 

We shall use the notation $(\l,\e)=1^{c(1)}_{\e(1)}\oplus 2^{c(2)}_{\e(2)}\oplus\cdots$.

In \cite{spal}, \S 2.8, 2.9, 2.10, there are formulas for the dimensions of centralizers of unipotent elements in $Sp(2n)$, $O(2n)$ (hence also in $SO(2n)$), the determination of the component groups $C(u)/C(u)^\circ$ in the various cases, and an explicit definition of a partial order on pairs $(\l,\e)$ such that $C_{\l,\e}\leq C_{\mu,\phi}\iff (\l,\e)\leq (\mu,\phi)$.

We shall use the notation as in \cite{spal}. As above mentioned, for the classical groups we only have to consider $p=2$, and then groups of type $C_n$ and $D_n$. For convenience, we shall work with $Sp(2n)$ and $SO(2n)$.

Strategy of the proof. Let $G_\C$ be the corresponding group over $\C$. We have shown in \cite{CCC} that for every spherical conjugacy class $\cal C$ of $G_\C$ there exists an involution $w=w(\cal C)$ in $W$ such that $\dim {\cal C}=\ell(w)+\rk(1-w)$, with ${\cal C}\cap BwB\not=\emptyset$. For each group $G$ we introduce a certain set $\O(G)$ of unipotent conjugacy classes which are candidates for being spherical. For each $\O\in\O(G)$ we show that there is a (non-necessarily unipotent) spherical conjugacy class ${\cal C}$ in $G_\C$ such that
$$
\dim \O=\dim{\cal C}
$$
Let $w=w(\cal C)$. Our aim is to show that $\O\cap BwB\not=\emptyset$.

\begin{definition}\label{sferica} Let $\O$ be a conjugacy class of $G$. We say that $\O$ satisfies $(*)$ if there exists $w\in W$ such that $BwB\cap \O\not=\emptyset$ and $\dim \O=\ell(w)+\rk(1-w)$.
\end{definition}

Let $\O$ be a conjugacy class in $G$. There exists a unique
element $w=w(\O)
\in W$
such that $\O\cap BwB$ is open dense in $\O$. In particular 
\begin{equation}
\overline{\O}=\overline {\O\cap
Bw B}\subseteq \overline {Bw B}.
\label{inclusion}
\end{equation}

It follows that if $y$ is an element of $\overline\O$ and $y\in Bw B$, then
$w\leq z$ in the
Chevalley-Bruhat order of $W$.

We recall the following result proved in \cite{CCC}, Theorem 5 over $\C$, but which is valid with the same proof over any algebraically closed field.

\begin{theorem}\label{metodo}
Suppose that $\O$ contains an element $x\in B{w}B$. 
Then $$\dim B.x
\geq \ell(w)+\rk(1-w).$$ In
particular $\dim {\O} \geq \ell(w)+\rk(1-w)$.
If, in addition, $\dim\O\leq\ell(w)+\rk (1-w)$ then $\O$ is spherical, 
$w=w(\O)$ and $B.x$ is the dense $B$-orbit in $\O$.\cvd
\end{theorem}

Let $\O$ be a conjugacy class of $G$ and let  $w=w(\O)$. If $\O^{-1}=\O$ (i.e. if any element $x\in\O$ is conjugate to its inverse), then $w^2=1$. It is well known that over any algebraically closed field any unipotent element is conjugate to its inverse ( \cite{Cuno}, Lemma 1.16, \cite{Cdue}, Lemma 2.3. See also \cite{Lus}, Proposition 2.5 (a)), so that $w$ is an involution for every non-trivial unipotent conjugacy class $\O$. However, it has recently been shown in \cite{intersezioni}, that $w^2=1$  for every conjugacy class $\O$ in $G$.

If $g$ is in $Z(G)$, then $g\in T$, $\O_g=\{g\}$, $w(\O)=1$. In the remaining of the paper we shall consider only non-central conjugacy classes.

We shall use the following result

\begin{lemma}\label{scambio} Assume the positive roots $\b_1,\ldots,\b_\ell$ are such that $[X_{\pm \b_i},X_{\pm \b_j}]=1$ for every $1\leq i<j\leq \ell$. 
Then, for $\xi_1,\ldots,\xi_\ell\in \K^\ast$,  $g=x_{\b_1}(-\xi_1^{-1})\cdots x_{\b_\ell}(-\xi_\ell^{-1})$, $h=h_{\b_1}(-\xi_1)\cdots h_{\b_\ell}(-\xi_\ell)$ we have
$$
gx_{-\b_1}(\xi_1)\cdots x_{-\b_\ell}(\xi_\ell)g^{-1}=
n_{\b_1}\cdots n_{\b_\ell}hx_{\b_1}(2\xi_1^{-1})\cdots x_{\b_\ell}(2\xi_\ell^{-1})
$$
\end{lemma}
\pf By (\ref{relazioni}) we have
$
x_\a(-\xi^{-1})x_{-\a}(\xi)x_\a(\xi^{-1})=n_\a h_\a(-\xi)x_\a(2\xi^{-1})
$. 
The result follows from $[X_{\pm \b_i},X_{\pm \b_j}]=1$ for every $1\leq i<j\leq \ell$. \cvd

The hypothesis of the Lemma are satisfied for instance if $\b_1,\ldots,\b_\ell$ are pairwise orthogonal and long, as in \cite{Cos}, Lemma 4.1. In characteristic 2, we have
$[X_{\gamma},X_{\delta}]=1$ for every pair $(\gamma,\delta)$ of orthogonal roots, so that for any set of pairwise orthogonal roots $\b_1,\ldots,\b_\ell$ and for $g=x_{\b_1}(1)\cdots x_{\b_\ell}(1)$ we get
\begin{equation}
gx_{-\b_1}(1)\cdots x_{-\b_\ell}(1)g^{-1}=
n_{\b_1}\cdots n_{\b_\ell}\quad \text{for $p=2$}.
\label{p=2}
\end{equation}
 
\subsection{Classical groups in characteristic 2.}

As mentioned above, for the classical groups we have deal only with $Sp(2n)$ and $SO(2n)$. However, for completeness, we shall also deal with the case when $G$ has type $A_n$, since this is dealt with in \cite{FR}, but not in \cite{Giovanna-good}.

\subsubsection{Type $A_{n}$, $n\geq 1$.}

We show that every spherical unipotent conjugacy class satisfies $(*)$. The  spherical nilpotent orbits (and therefore the spherical unipotent classes) have been classified in \cite{FR}, and it follows that a unipotent conjugacy class $\O$ is spherical if and only if $\O=X_i$ the unipotent class $2^i\oplus 1^{n+1-2i}$ for $i=1,\ldots,m=\left[\frac{n+1}{2}\right]$. For every $i=1,\ldots,m$, let $\b_i=e_{i}-e_{n+2-i}$: then, as for $\C$ (and for any algebraically closed field of odd or zero characteristic)
the element $u=x_{-\b_1}(1)\cdots x_{-\b_i}(1)$ lies in $X_i\cap BwB$, where $w=s_{\b_1}\cdots s_{\b_i}$ is such that $\dim X_i=\ell(w)+\rk(1-w)$. In fact one may take $n_{\b_1}\cdots n_{\b_i}\in s_{\b_1}\cdots s_{\b_i}B\cap X_i$.

\begin{center}
\vskip-20pt
$$
\begin{array}{|c||c|c|}
\hline
\O  &w(\O)&x\in \O\cap Bw(\O)B \\ 
\hline
\hline
\begin{array}{c}
X_\ell=2^\ell\oplus 1^{n+1-2\ell}\\
\ell=1,\ldots,m=\left[\frac{n+1}{2}\right]
\end{array}
&\quad \displaystyle s_{\b_1}\cdots s_{\b_\ell}\quad&\quad\displaystyle n_{\b_1}\cdots n_{\b_\ell}\quad \\
\hline
\end{array}
$$
\end{center}
\begin{center} Table \totable:  Spherical unipotent classes in $A_n$ ($p=2$).
\end{center}

In particular, a non-trivial unipotent class $\O$ is spherical if and only if it consists of involutions, if and only if $\O$ has a representative of the form $x_{\gamma_1}(1)\cdots x_{\gamma_t}(1)$, where $\gamma_1,\ldots,\gamma_t$ are pairwise orthogonal simple roots.

\medskip

\subsubsection{Type $C_n$ (and $B_n$), $n\geq 2$.}

We first show that if $u$ is an involution of $G$, then $\O_u$ is spherical, by showing that $\O_u$ satisfies $(*)$. So let $u$ be an involution of $G=Sp(2n)$. Then the partition corresponding to $u$ is of the form $\l=1^{c_1}\oplus 2^{c_2}$, with $c_2=\ell$, $c_1=2n-2\ell$,

Using the above recalled description of unipotent conjugacy classes,
let $\l=2^{i}\oplus 1^{2n-2i}$, for $i=1,\ldots,n$. Then we have $\e_0=1$, $\e_1=\omega$, $\e_i=0$ for $i\geq 3$. As for $\e_2$, we have $\e_2=1$ if $i$ is odd. On the other hand, if $i$ is even, we have both possibilities $\e_2=0$ or $1$. We denote by $X_i$ the class corresponding to $\e_2=1$, and by $Y_i$ the class corresponding to $\e_2=0$ (when $i$ is even)

 We denote by $X_{\ell,\C}$ the unipotent class in $Sp(2n,\C)$ corresponding to $\l=2^{\ell}\oplus 1^{2n-2\ell}$. Then we get
$$
\begin{array}{llllll}
\dim X_\ell&=&\dim X_{\ell,\C}&=&\ell(2n-\ell+1)&\ell\in\{1,\ldots,n\} \\
\dim Y_\ell&= &\dim X_{\ell, \C}-\ell&=&2n\ell-\ell^2=\ell(2n-\ell)&\ell\in\{1,\ldots,n\},\ \text{$\ell$ even}
\end{array}
$$

Note that if $\ell$ is even and we write $\ell=2\ell'$, then $\dim Y_\ell=\dim \O_{\s_{\ell'},\C}$, where $\O_{\s_{\ell'},\C}$ is the conjugacy class in $Sp(2n,\C)$ of the involution $\s_{\ell'}$ (\cite{CCC}, Table 1). In $Sp(2n,\C)$, the spherical semisimple conjugacy class $\O_{\s_{\ell'}}$ lies over $w=s_{\g_1}\cdots s_{\g_{\ell'}}$ (\cite{CCC}, Table 5, \cite{Cos}, \S 4.2.2). 

We observe that if the partition associated to the involution $u$ is $\l=2^{2\ell}\oplus 1^{2n-4\ell}$, then $\O_u=2^{2\ell}_0\oplus 1^{2n-4\ell}$ if and only if 
 $f((u-1)v,v)=0$ for every $v\in V$ (here $f$ is the bilinear form on $V$ used to define $Sp_{2n}(k)$). Let $w$ be an involution of $W$, $L(w)=\{\beta\in \Phi^+\mid w(\beta)=-\beta, \beta\ \text{long}\}$, $L(w)_\perp=\{\gamma\in \Phi^+\mid w(\gamma)=-\gamma, (\gamma,L(w))=0, \gamma\ \text{short}\}$. Then $w=\displaystyle\prod_{\beta\in L(w)} s_\beta\prod_{\gamma\in L(w)_\perp} s_\gamma$. Let $x=\displaystyle\prod_{\beta\in L(w)} n_\beta\prod_{\gamma\in L(w)_\perp} n_\gamma$. Then $x$ is an involution in $BwB$ and the number of blocks of length 2 in the Jordan canonical form of $x$ is $\o{L(w)}+2\o{L(w)_\perp}$. If this number is even, then $f((x-1)v,v)=0$ for every $v\in V$ if and only if $L(w)=\emptyset$.

 We put
$\b_i=2e_{i}$ for $i=1,\ldots,n$
 and $\g_i=e_{2i-1}+e_{2i}$ for $\ell=1,\ldots, p=[\frac n2]$.

Then it is straightforward to show that 
$$
x_{-\b_1}(1)\cdots x_{-\b_\ell}(1)\in X_\ell\cap Bs_{\b_1}\cdots s_{\b_\ell}B\cap U^-\quad\text{for $\ell=1,\ldots,n$}
$$
$$
x_{-\g_1}(1)\cdots x_{-\g_\ell}(1)\in Y_{2\ell}\cap Bs_{\g_1}\cdots s_{\g_\ell}B\cap U^-\quad\text{for $\ell=1,\ldots,p$}
$$
By (\ref{p=2}), we can choose
$$
n_{\b_1}\cdots n_{\b_\ell}\in X_\ell\cap wB
\quad\text{for $\ell=1,\ldots,n$.}
$$
$$
n_{\g_1}\cdots n_{\g_\ell}\in Y_{2\ell}\cap wB
\quad
\text{for $\ell=1,\ldots,p$.}
$$

One can easily deduce which classes of involutions have a representative of the form $u=\prod_{i\in K}x_{\a_i}(1)$ for a certain subset $K$ of $\Pi$. Note that since $u$ is an involution, then $(\a_i,\a_j)=0$ if $i$, $j\in K$ with $i\not=j$. Up to the $W$ action, we have only the following subsets $K$, and the corresponding classes:
$$
\prod_{i=1}^{\ell}x_{\a_{n-2(i-1)}}(1)\in X_{2\ell-1}\quad\text{for $\ell=1,\ldots,\left[\frac{n+1}2\right]$}
$$
$$
\prod_{i=1}^{\ell}x_{\a_{2i-1}}(1)\in Y_{2\ell}\quad\text{for $\ell=1,\ldots,p=\left[\frac{n}2\right]$}
$$
These exhaust the conjugacy classes of involutions with representative of the form $\prod_{i\in K}x_{\a_i}(1)$. In particular all $X_{2i}$ have no representative of the form $\prod_{i\in K}x_{\a_i}(1)$.
The point is that in good characteristic, for $\ell=1,\ldots,p=\left[\frac{n}2\right]$ the element
$\prod_{i=1}^{\ell}x_{\a_{2i-1}}(1)$ is conjugate to 
$\prod_{i=1}^{2\ell}x_{\b_{i}}(1)$ (which lies in $X_{2\ell}$).

\medskip

If $J_i=\{i+1,\ldots,n\}$ ($J_n=\vuoto$) for $i=1,\ldots, n$, $K_\ell=\{1,3,\ldots,2\ell-1,2\ell+1,2\ell+2,\ldots,n\}$
for $\ell=1,\ldots, p$, we obtained

\begin{center}
\vskip-20pt
$$
\begin{array}{|c||c|c|c|c|}
\hline
\O  &J& w(\O) & x\in \O\cap  Bw(\O)B&\dim \O \\ 
\hline
\hline 
\begin{array}{c}
X_\ell=2^{\ell}\oplus 1^{2n-2\ell}
\\
\ell=1,\ldots,n
\end{array} &\displaystyle J_\ell & s_{\b_1}\cdots s_{\b_\ell} &n_{\b_1}\cdots n_{\b_\ell}&\ell(2n-\ell+1)\\
\hline
\begin{array}{c}
Y_{2\ell}=2^{2\ell}_0\oplus 1^{2n-4\ell}
\\
\ell=1,\ldots,p=[\frac n2]
\end{array}
&  \displaystyle K_{\ell}& \quad\displaystyle s_{\g_1}\cdots s_{\g_\ell}\quad& n_{\g_1}\cdots n_{\g_\ell}&4\ell(n-\ell)
\\
\hline
\end{array}
$$
\end{center}
\begin{center} Table \totable:  Involutions in $C_n$,
$n\geq 2$, $p=2$.
\end{center}

where $w(\O)=w_0\wJ$. By Theorem \ref{metodo}, we have proved

\begin{proposition}\label{sfericheC_n} Let $\O$ be the conjugacy class of an involution of $Sp(2n)$ in characteristic 2. Then $\O$ is spherical.\cvd
\end{proposition}

Our aim is to show that a (non-trivial) unipotent conjugacy class $\O_u$ is spherical if and only if $u$ is an involution. By \cite{FR}, Remark 2.14, the orbit $\O$ is spherical if and only if $G/H$ is spherical, where $H$ is the isotropy subgroup of an element in $\O$. Moreover $G/H$ is spherical if and only if $G/H^\circ$ is spherical, where $H^\circ$ denotes the connected component of $H$. We shall therefore use the following 

\begin{lemma}\label{generale} Let $\O$ be a $G$-orbit with isotropy subgroup $H$. Then $\O$ is spherical if and only if $G/H^\circ$ is spherical.\cvd
\end{lemma}

By Proposition \ref{sfericheC_n}, we are left to show that if the (non-trivial) unipotent class $\O$ does not consist of involutions, then $\O$ is not spherical. Let $u$ be a unipotent element of order greater than 4, and let $v$ be an element of order 4 in the subgroup generated by $u$. Since $C(u)\leq C(v)$, if $\O_v$ is non-spherical, then also $\O_u$ is non-spherical. We are therefore left to consider the set $X$ of conjugacy classes of elements of order 4. By \cite{Knop}, Theorem 2.2, it is enough to show that the minimal elements in $X$ are not spherical. 

From the explicit definition of a partial order on pairs $(\l,\e)$ such that $C_{\l,\e}\leq C_{\mu,\phi}\iff (\l,\e)\leq (\mu,\phi)$ given in
 \cite{spal}, \S 2.10, it follows that the minimal elements in $X$ are the classes $3^2\oplus 1^{2n-6}$ (if $n\geq 3$) and $4\oplus 1^{2n-4}$.
 
 In the following lemma we deal with these cases. We also consider a case in $D_n$.

\medskip
\begin{lemma}\label{nonsfericheC_n} Let $\O$ be the unipotent conjugacy class of type $3^2\oplus 1^{2n-6}$ in $C_n$ or $D_n$. Then $\O$ is not spherical ($\ch k=2$).
\end{lemma}
\pf
Let $u$ be an element in $\O=3^2\oplus 1^{2n-6}$ (this exists in $C_n$ if $n\geq 3$). In $Sp(2n)$, we may take $u=x_{\a_2}(1)x_{\gamma-\a_2}(1)$, where $\gamma$ is the highest short root ($\gamma=e_1+e_2$), and get $C(u)^\circ\leq P$ where $P$ is the maximal parabolic subgroup $P_{I\setminus\{\a_2\}}$. Then $C(u)^\circ=CR$,  $C=H_{\a_1}\times K$, where $K$ is the $C_{n-3}$-subgroup of $G$ corresponding to $\{\a_4,\ldots,\a_n\}$, and $R$ is the subgroup\
$$
\{u\in R_u P\mid u=\prod_{\a\in\Phi^+} x_{\a}(z_\a)\mid z_{a_2}=z_{\gamma-\a_2}\}
$$
of codimension 1 in $R_u P$. It follows that $C(u)^\circ$ fixes the element $e_{\a_2}+e_{\gamma-\a_2}$ of $\mathfrak g$. However, we clearly have $C(u)\leq C(u^2)$, and $u^2=x_{\gamma}(1)$. But $C(x_{\gamma}(1))$ fixes the element $e_{\gamma}$ of $\mathfrak g$, since $x_{\gamma}(1)=1+e_{\gamma}$ in $M_{2n}(k)$. It follows that $C(u)^\circ$ has 2 linearly independent invariants in $\mathfrak g$, so that $Sp(2n)/C(u)^\circ$ is not spherical.

Since $u$ lies in $SO(2n)$, and both $e_{\a_2}+e_{\gamma-\a_2}$ and $e_{\gamma}$ are in the Lie algebra of $SO(2n)$, the $SO(2n)$-orbit of $u$ is not spherical as well.
\cvd

\begin{lemma}\label{nonsferiche4} Let $\O$ be the unipotent conjugacy class of type $4\oplus 1^{2n-4}$ in $C_n$. Then $\O$ is not spherical ($\ch k=2$).
\end{lemma}
\pf
Let $u$ be an element in $\O=4\oplus 1^{2n-4}$. In $Sp(2n)$, we may take $u=x_{\a_1}(1)x_{\delta}(1)$, where $\delta=2e_2$, and get $C(u)^\circ\leq P$, where $P$ is the parabolic subgroup $P_{I\setminus\{\a_1,\a_2\}}$. Then $C(u)^\circ=CR$, where $C$ is the $C_{n-2}$-subgroup of $G$ corresponding to $\{\a_3,\ldots,\a_n\}$, and $R$ is a subgroup of $U$.
In fact $\dim R=2n-2$, and $R$ is the product of $X_\a$'s, where 
$\a=e_1\pm e_i$, $i=3,\ldots,n$ or $\a=e_1+e_2$ or $\a=2e_1$.

It follows that $C(u)^\circ$ fixes the first 2 basis vectors $v_1$ and $v_{2}$ of the natural module of $G$, so that $Sp(2n)/C(u)^\circ$ is not spherical.\cvd

We have proved

\begin{proposition}\label{sfericheC_n_completo} Let $\O$ be a non-trivial unipotent conjugacy class of $Sp(2n)$ in characteristic 2. Then $\O$ is spherical if and only if it consists of involutions.\cvd
\end{proposition}

\medskip

\subsubsection{Type $D_n$, $n\geq 4$.}

Let $m=\left[\frac{n}2\right]$. We put $\b_i=e_{2i-1}+e_{2i}$, $\d_i=e_{2i-1}-e_{2i}$ for $i=1,\ldots, m$.
For $\ell=1,\ldots,m-1$ we put $J_\ell=\{2\ell+1,\ldots,n\}$, $J_m=\vuoto$, 
$K_\ell=J_\ell\cup\{1,3,\ldots,2\ell-1\}$ for $\ell=1,\ldots,m$. Also we put $K_m'=\{1,3,\ldots,n-3,n\}$.

Let $u$ be an involution of $G=SO(2n)$. Then the partition corresponding to $u$ is of the form $\l=2^{c_2}\oplus 1^{c_1}$, with $c_2=2i$, $c_1=2n-4i$, $i=1,\ldots,m$.

For each $\ell=1,\ldots,m-1$ there are 2 conjugacy classes corresponding to $\l=2^{2i}\oplus 1^{2n-4i}$: we denote by $X_\ell$ the class $2^{2\ell}_0\oplus 1^{2n-4\ell}$, and by $Z_\ell$ the class $2^{2\ell}\oplus 1^{2n-4\ell}$. If $\ell=m$, then we denote by $Z_m$ the class $2^{2m}\oplus 1^{2n-4m}$. The conjugacy class in $O(2n)$ corresponding to  $2^{2m}_0\oplus 1^{2n-4m}$ is a single class $X_m$ in $SO(2n)$ if $n$ is odd, while it splits into 
 2 conjugacy classes $X_m$ and $X_m'$ in $SO(2n)$ if $n$ is even.

We have
$$
\dim Z_\ell=4\ell(n-\ell)\quad,\quad
\dim X_\ell=2\ell(2n-2\ell-1)\quad\text{
for $\ell=1,\ldots,m$}
$$
with 
$\dim X_m'=\dim X_m$ if $n$ is even.

We have chosen the notation so that for each conjugacy class of involutions $\O$ in $G$, the conjugacy class $\cal C$ in $G_\C$ denoted by the same symbol in \cite{Cos} \S 4.3, has the same dimension. For the corresponding $w$, we write $w$ as a product of commuting reflections, $w=s_{\gamma_1}\cdots s_{\gamma_t}$. It is straightforward to prove that in each case the element $x=n_{\gamma_1}\cdots n_{\gamma_t}$ lies in $\O$. We summarize in the following tables the results obtained:
\begin{center}
\vskip-20pt
$$
\begin{array}{|c||c|c|c|c|}
\hline
\O  &J& w(\O) & x\in \O\cap Bw(\O)B&\dim \O \\ 
\hline
\hline 
\begin{array}{c}
Z_\ell=2^{2\ell}\oplus 1^{2n-4\ell}
\\
\ell=1,\ldots,m
\end{array} &\displaystyle J_\ell & s_{\b_1}s_{\d_1}\cdots s_{\b_\ell}s_{\d_\ell}&n_{\b_1}n_{\d_1}\cdots n_{\b_\ell}n_{\d_\ell}&4\ell(n-\ell)\\
\hline
\hline
\begin{array}{c}
X_{\ell}=2^{2\ell}_0\oplus 1^{2n-4\ell}
\\
\ell=1,\ldots,m
\end{array}
&  \displaystyle K_{\ell}& \quad\displaystyle s_{\b_1}\cdots s_{\b_\ell}\quad& n_{\b_1}\cdots n_{\b_\ell}&2\ell(2n-2\ell-1)
\\
\hline
\begin{array}{c}
X_{m}'=(2^{2m}_{0})'
\end{array}
&  \displaystyle K_m'& \quad\displaystyle s_{\b_1}\cdots s_{\b_{m-1}}s_{\a_{n-1}}\quad& n_{\b_1}\cdots n_{\b_{m-1}}n_{\a_{n-1}}&n(n-1)\\
\hline
\end{array}
$$
\end{center}
\begin{center} Table \totable:  Involutions in $D_n$,
$n\geq 4$, $n=2m$.
\end{center}

\begin{center}
\vskip-20pt
$$
\begin{array}{|c||c|c|c|c|}
\hline
\O  &J& w(\O) & x\in \O\cap Bw(\O)B&\dim \O \\ 
\hline
\hline 
\begin{array}{c}
Z_\ell=2^{2\ell}\oplus 1^{2n-4\ell}
\\
\ell=1,\ldots,m
\end{array} &\displaystyle J_\ell & s_{\b_1}s_{\d_1}\cdots s_{\b_\ell}s_{\d_\ell}&n_{\b_1}n_{\d_1}\cdots n_{\b_\ell}n_{\d_\ell}&4\ell(n-\ell)\\
\hline
\hline
\begin{array}{c}
X_{\ell}=2^{2\ell}_0\oplus 1^{2n-4\ell}
\\
\ell=1,\ldots,m
\end{array}
&  \displaystyle K_{\ell}& \quad\displaystyle s_{\b_1}\cdots s_{\b_\ell}\quad& n_{\b_1}\cdots n_{\b_\ell}&2\ell(2n-2\ell-1)
\\
\hline
\end{array}
$$
\end{center}
\begin{center} Table \totable:  Involutions in $D_n$,
$n\geq 4$, $n=2m+1$.
\end{center}

By Theorem \ref{metodo}, we have proved

\begin{proposition}\label{sfericheD_n} Let $\O$ be the conjugacy class of an involution of $SO(2n)$ in characteristic 2. Then $\O$ is spherical.\cvd
\end{proposition}

Our aim is to show that a (non-trivial) unipotent conjugacy class $\O_u$ is spherical if and only if $u$ is an involution. Using the same arguments as in case $C_n$, we are left to consider the set $X$ of conjugacy classes of elements of order 4, and then show that the minimal elements in $X$ are not spherical. From the explicit definition of a partial order on pairs $(\l,\e)$ such that $C_{\l,\e}\leq C_{\mu,\phi}\iff (\l,\e)\leq (\mu,\phi)$ given in
 \cite{spal}, \S 2.10, it follows that the minimal element in $X$ is the class $3^2\oplus 1^{2n-6}$. By Lemma \ref{nonsfericheC_n}, this class is not spherical. We have therefore proved 
 
 \begin{proposition}\label{sfericheD_n_completo} Let $\O$ be a non-trivial unipotent conjugacy class of $SO(2n)$ in characteristic 2. Then $\O$ is spherical if and only if it consists of involutions.\cvd
\end{proposition}

\begin{remark}
{\rm 
From our discussion, it follows that for $D_n$ the map $\pi_G:X^{G'}\to X^G$ defined in \cite{spal}, Theorem III.5.2 induces an isomorphism of p.o. sets between $X^{G'}_{\sph}\to X^G_{\sph}$ where $X^{G'}_{\sph}$, $X^G_{\sph}$ are the corresponding sets of spherical unipotent classes. In particular, every spherical unipotent conjugacy class has a representative of the form $\prod_{\a\in K}x_{\a}(1)$ for a certain set of pairwise orthogonal simple roots $K$.
}
\end{remark}

\medskip
 
\subsection{Exceptional groups.}

For the exceptional groups, we use \cite{LLS}, Table 2. In this table, for each group $G$, there are all unipotent conjugacy classes $\O_u$, in every characteristic, for which the dimension of $C(u)$ is greater than a certain number $l_G$. From this we deduce the following table

\begin{center}
\vskip-20pt
$$
\begin{array}{|c|c|c|c|c|}
\hline
G  & \dim B & \text{ $u$ with $\dim \O_u\leq \dim B$}&\dim \O_u& \o{C(u)/C(u)^\circ}  \\ 
\hline
\hline
E_6
& 42&
\begin{array}{c}
A_1,\
2A_1,\
3A_1,\
A_2
\end{array}&
22, 32, 40, 42&
1, 1, 1, 2\\
\hline
E_7
& 70&
\begin{array}{c}
A_1,\
2A_1,\
3A_1'',\
3A_1',\
A_2,\
4A_1
\end{array}&
34, 52, 54, 64, 66, 70&
1, 1, 1, 1, 2,1\\
\hline
E_8
& 128&
\begin{array}{c}
A_1,\
2A_1,\
3A_1,\
A_2,\
4A_1
\end{array}&
58, 92, 112, 114, 128&
1, 1, 1, 2, 1\\
\hline
F_4
& 28&
\begin{array}{c}
A_1,\
\tilde A_1 (p=2),\
\tilde A_1 (p\not=2),\
\tilde A_1^{(2)} (p=2),\
A_1\tilde A_1
\end{array}&
16, 16, 22, 22, 28&
1, 1, 2, 1, 1\\
\hline
G_2
& 8&
\begin{array}{c}
A_1,\
\tilde A_1 (p=3),\
\tilde A_1 (p\not=3),\
\tilde A_1^{(3)} (p=3)
\end{array}&
6, 6, 8, 8&
1, 1, 1, 1\\
\hline
\end{array}
$$
\end{center}
\begin{center} Table \totable\label{tavolauno}: Unipotent classes of small dimension in exceptional groups.
\end{center}
where the value of $\o{C(u)/C(u)^\circ}$ for $E_7$ refers to the adjoint group (see \cite{Ale}).

The unipotent conjugacy classes appearing in Table \ref{tavolauno} are the only candidates to being spherical. We shall show that they are all spherical, except for the classes of type $A_2$ in $E_6$, $E_7$ and $E_8$. Note that when $p=2$, then all classes of involutions appear in Table \ref{tavolauno} by the results in \cite{AS}.

\medskip
\begin{lemma}\label{nonsferiche} Let $\O$ be the unipotent conjugacy class of type $A_2$ in $E_6$, $E_7$ or $E_8$. Then $\O$ is not spherical (in any characteristic).
\end{lemma}
\pf
Let $u$ be an element in $\O$. From \cite{LLS}, Table 2, it follows that the type of $C(u)^\circ$ is independent of the characteristic. For completeness, we determine $C(u)^\circ$ in all cases.

In $E_8$, we may take $u=x_{\a_8}(1)x_{\beta-\a_8}(1)$, where $\beta$ is the highest root, and get $C(u)^\circ\leq P$ where $P$ is the maximal parabolic subgroup $P_{I\setminus\{\a_8\}}$. Then $C(u)^\circ=CR$, where $C$ is the $E_6$-subgroup of $G$ corresponding to $\{\a_1,\ldots,\a_6\}$, and 
$$
R=\{g=\prod_{\a\in\Phi^+} x_\a(k_\a)\in R_uP\mid k_{\a_8}=k_{\beta-\a_8}\}.
$$

In $E_7$, we may take $u=x_{\a_1}(1)x_{\beta-\a_1}(1)$, where $\beta$ is the highest root, and get $C(u)^\circ\leq P$ where $P$ is the maximal parabolic subgroup $P_{I\setminus\{\a_1\}}$. Then $C(u)^\circ=CR$, where $C$ is the $A_5$-subgroup of $G$ corresponding to $\{\a_2,\a_4,\a_5,\a_6,\a_7\}$, and 
$$
R=\{g=\prod_{\a\in\Phi^+} x_\a(k_\a)\in R_uP\mid k_{\a_1}=k_{\beta-\a_1}\}.
$$

In $E_6$, we may take $u=x_{\a_2}(1)x_{\beta-\a_2}(1)$, where $\beta$ is the highest root, and get $C(u)^\circ\leq P$ where $P$ is the maximal parabolic subgroup $P_{I\setminus\{\a_2\}}$. Then $C(u)^\circ=CR$, where $C$ is the $A_2\times A_2$-subgroup of $G$ corresponding to $\{\a_1,\a_3,\a_5,\a_6\}$, and 
$$
R=\{g=\prod_{\a\in\Phi^+} x_\a(k_\a)\in R_uP\mid k_{\a_2}=k_{\beta-\a_2}\}.
$$

It is well known that the class $A_2$ is not spherical in $E_6$, $E_7$ or $E_8$ over any algebraically closed field of characteristic zero. We may now apply \cite{Brundan}, Theorem 2.2 (i). Note that the groups $C(u)^\circ$ involved are all defined over $\Z$, and the argument in the proof of \cite{Brundan}, Theorem 2.2 (i) is valid in our situation. Therefore $G/C(u)^\circ$ is not spherical in any positive characteristic. It follows that $\O_u$ is not spherical by Lemma \ref{generale}.\cvd

\subsubsection{Type $E_6$.}

We put
$$
\begin{array}{cclll}
\b_1 = (1,2,2,3,2,1),&
\b_2 =(1,0,1,1,1,1)\\
\b_3 =(0,0,1,1,1,0),&
\b_4 = (0,0,0,1,0,0)
\end{array}
$$
\medskip
For groups of type $E_6$ we have to consider $p=2$, $3$. If $p=3$, then we may apply the arguments in \cite{CCC}, Theorem 13 to prove that the orbits of type $A_1$, $2A_1$ and $3A_1$ satisfy $(*)$, hence are spherical, since to handle $3A_1$ we need results for $D_4$ which holds due to
\cite{Giovanna-good}, Theorem 3.4 and its proof (in fact what we need is that the maximal spherical unipotent conjugacy class $\O'$ of $D_4$ satisfies $(*)$ when $p=3$). So now assume $p=2$. Then again we may use the proof of \cite{CCC}, Theorem 3.4 to deal with $A_1$ and $2A_1$. Note that in these cases
$$
x_{-\beta_1}(1)\in A_1\cap Bs_{\beta_1} B\cap U^-\quad,\quad
x_{-\beta_1}(1)x_{-\beta_2}(1)\in 2A_1\cap Bs_{\beta_1}s_{\beta_2}B\cap U^-
$$
with $x_{-\beta_1}(1)\sim n_{\beta_1}$, $x_{-\beta_1}(1)x_{-\beta_2}(1)\sim n_{\beta_1}n_{\beta_2}$. To deal with $3A_1$, we still may use the arguments in the proof of \cite{CCC}, Theorem 3.4 since we have shown in \S 3.1.3 that the maximal spherical unipotent conjugacy class $\O'$ of $D_4$ satisfies $(*)$ when $p=2$, or directly observe that $x=n_{\beta_1}n_{\beta_2}n_{\beta_3}n_{\a_4}$ is an involution in $Bn_{\beta_1}n_{\beta_2}n_{\beta_3}n_{\a_4}B=Bw_0B$. Then $\dim \O_x\geq \ell(w_0)+\rk(1-w_0)=40$ by Theorem \ref{metodo}, so that $x\in 3A_1$ by Table \ref{tavolauno}, since elements in $A_2$ are not involutions. We have proved

\begin{proposition}\label{unip_E_6} Let $\O$ be a non-trivial unipotent conjugacy class in $E_6$. Then $\O$ is spherical if and only if $\O=A_1, 2A_1$ or $3A_1$. 
In each case $\O$ satisfies $(*)$. If $\ch k=2$, these are precisely the classes consisting of involutions.
\end{proposition}

\subsubsection{Type $E_7$.}

 We put
$$
\begin{array}{lll}
\b_1 = (2,2,3,4,3,2,1),\
\b_2 =(0,1,1,2,2,2,1),\
\b_3 =(0,1,1,2,1,0,0),\\
\b_4 = \a_7,\quad
\b_5=\a_5,\quad
\b_6=\a_3,\quad
\b_7=\a_2
\end{array}
$$
For groups of type $E_7$ we have to consider $p=2$, $3$. If $p=3$, then we may apply the arguments in \cite{CCC}, Theorem 13 to prove that the orbits of type $A_1$, $2A_1$, $(3A_1)'$, $(3A_1)''$ and $4A_1$ are spherical, since we need results for $D_n$ which holds due to
\cite{Giovanna-good}, Theorem 3.4 and its proof. So now assume $p=2$. Then again we may use the proof of \cite{CCC}, Theorem 3.4 to deal with $A_1$, $2A_1$ and  $(3A_1)''$. Note that in these cases
$$
x_{-\beta_1}(1)\in A_1\cap Bs_{\beta_1} B\cap U^-\quad,
$$
$$
x_{-\beta_1}(1)x_{-\beta_2}(1)\in 2A_1\cap Bs_{\beta_1}s_{\beta_2}B\cap U^-\quad,
$$
$$
x_{-\beta_1}(1)x_{-\beta_2}(1)x_{-\a_7}(1)\in (3A_1)''\cap Bs_{\beta_1}s_{\beta_2}s_{\a_7}B\cap U^-
$$
To deal with $(3A_1)'$ and $4A_1$,  again we may apply the arguments in \cite{CCC}, Theorem 13, since we need results for $D_n$ when $p=2$ which we proved in \S 3.1.3. However, it is also possible to show directly that 
$n_{\beta_1}n_{\beta_2}n_{\beta_3}n_{\a_3}\in s_{\beta_1}s_{\beta_2}s_{\beta_3}s_{\a_3}B\cap (3A_1)'$.
To deal with $4A_1$ one can observe that $x=n_{\beta_1}\cdots n_{\b_7}$ is an involution in $Bw_0B$. Then $\dim \O_x\geq \ell(w_0)+\rk(1-w_0)=70$ by Theorem \ref{metodo}, so that $x\in 4A_1$ by Table \ref{tavolauno}. We have proved

\begin{proposition}\label{unip_E_7} Let $\O$ be a non-trivial unipotent conjugacy class in $E_7$. Then $\O$ is spherical if and only if $\O=A_1, 2A_1$,  $(3A_1)'$, $(3A_1)''$ or $4A_1$. 
In each case $\O$ satisfies $(*)$. If $\ch k=2$, these are precisely the classes consisting of involutions.
\end{proposition}

\subsubsection{Type $E_8$.}

We put
$$
\begin{array}{lll}
&\b_1 =   (2,3,4,6,5,4,3,2),\ 
\b_2=(2,2,3,4,3,2,1,0),\
\b_3=(0,1,1,2,2,2,1,0),\\
&\b_4=(0,1,1,2,1,0,0,0),\
\b_5=\a_7, \
\b_6=\a_5,\
\b_7=\a_3,\
\b_8=\a_2.
\end{array}
$$
For groups of type $E_8$ we have to consider $p=2$, $3$, $5$. If $p=3$ or $5$, then we may apply the arguments in \cite{CCC}, Theorem 13 to prove that the orbits of type $A_1$, $2A_1$, $3A_1$ and $4A_1$ are spherical, since to handle $3A_1$ and $4A_1$ we need results for $D_4$ and $D_6$ which holds due to
\cite{Giovanna-good}, Theorem 3.4 and its proof. So now assume $p=2$. Then again we may use the proof of \cite{CCC}, Theorem 3.4 to deal with $A_1$ and $2A_1$. Note that in these cases
$$
x_{-\beta_1}(1)\in A_1\cap Bs_{\beta_1} B\cap U^-\quad,
$$
$$
x_{-\beta_1}(1)x_{-\beta_2}(1)\in 2A_1\cap Bs_{\beta_1}s_{\beta_2}B\cap U^-\quad,
$$
To deal with $3A_1$ and $4A_1$,  again we may apply the arguments in \cite{CCC}, Theorem 13, since we need results for $D_n$ when $p=2$ which we proved in \S 3.1.3. However, it is also possible to show directly that 
$n_{\beta_1}n_{\beta_2}n_{\beta_3}n_{\b_5}\in s_{\beta_1}s_{\beta_2}s_{\beta_3}s_{\b_5}B\cap 3A_1$.
To deal with $4A_1$ one can observe that $x=n_{\beta_1}\cdots n_{\b_8}$ is an involution in $Bw_0B$. Then $\dim \O_x\geq \ell(w_0)+\rk(1-w_0)=128$ by Theorem \ref{metodo}, so that $x\in 4A_1$ by Table \ref{tavolauno}. We have proved

\begin{proposition}\label{unip_E_8} Let $\O$ be a non-trivial unipotent conjugacy class in $E_8$. Then $\O$ is spherical if and only if $\O=A_1, 2A_1$,  $3A_1$, or $4A_1$. 
In each case $\O$ satisfies $(*)$. If $\ch k=2$, these are precisely the classes consisting of involutions.
\end{proposition}

\subsubsection{Type $F_4$.}

We put
$$
\begin{array}{ll}
\b_1 =(2,3,4,2),&\b_2=(0,1,2,2),\\
\b_3=(0,1,2,0),&\b_4=(0,1,0,0)
\end{array}
$$
also $\gamma_1$ is the highest short root $(1,2,3,2)$.

For groups of type $F_4$ we have to consider $p=2$, $3$. If $p=3$, then we may apply the arguments in \cite{CCC}, Theorem 13 to prove that the orbits of type $A_1$, $\tilde A_1$, $A_1\tilde A_1$ and are spherical, since to handle $A_1\tilde A_1$ we need results for $D_4$  which holds due to
\cite{Giovanna-good}, Theorem 3.4 and its proof. So now assume $p=2$. Here there are more conjugacy classes $\O_u$ (due to the presence of the graph automorphism of $G$)   for which $\dim \O_u\leq \dim B$ (see Table \ref{tavolauno}). Each class consists of involutions. We may take the following representatives
$$
x_{-\beta_1}(1)\in A_1\cap Bs_{\beta_1} B\cap U^-\quad,
$$
$$
x_{-\gamma_1}(1)\in \tilde A_1\cap Bs_{\gamma_1}B\cap U^-\quad,
$$
To deal with $\tilde A_1^{(2)}$, we observe that $u=x_{\beta_1}(1)x_{\gamma_1}(1)\in \tilde A_1^{(2)}$ by \cite{AS}, (13.1).
Let $K$ be the $C_2$-subgroup of $G$ with basis $\{(1,1,1,0), \b_2\}$.
Then $\b_1$ and $\gamma_1$ are the highest long and short root in $K$ respectively. A direct calculation in $C_2$ shows that $u$ is conjugate to $v=x_{\beta_1}(1)x_{\b_2}(1)$, hence
$$
x_{-\beta_1}(1)x_{-\beta_2}(1)\in \tilde A_1^{(2)}\cap Bs_{\beta_1}s_{\beta_2}B\cap U^-
$$
Finally, to deal with  $A_1\tilde A_1$ we observe that $x=n_{\beta_1}\cdots n_{\b_4}$ is an involution in $Bw_0B$. Then $\dim \O_x\geq \ell(w_0)+\rk(1-w_0)=28$ by Theorem \ref{metodo}, so that $x\in A_1\tilde A_1$ by Table \ref{tavolauno}.

\begin{center}
\vskip-20pt
$$
\begin{array}{|c||c|c|c|}
\hline
\O  &J& w(\O) &\dim \O \\ 
\hline
\hline 
\begin{array}{c}
A_1
\\
\end{array} &\displaystyle  \{2,3,4\} & s_{\b_1}&16\\
\hline
\begin{array}{c}
\tilde A_1
\\
\end{array}
&  \displaystyle \{2,3\}& \quad\displaystyle s_{\beta_1}s_{\beta_2}\quad& 22
\\
\hline
\begin{array}{c}
A_1\tilde A_1
\end{array}
&  \displaystyle \emptyset& \quad\displaystyle w_0\quad&28\\
\hline
\end{array}
$$
\end{center}
\begin{center} Table \totable:  $F_4$, $p=3$ (or any $\ch k\not=2$).
\end{center}
\medskip

\begin{center}
\vskip-20pt
$$
\begin{array}{|c||c|c|c|c|}
\hline
\O  &J& w(\O) & x\in \O\cap Bw(\O)B&\dim \O \\ 
\hline
\hline 
\begin{array}{c}
A_1
\\
\end{array} &\displaystyle  \{2,3,4\} & s_{\b_1}&n_{\b_1}&16\\
\hline
\begin{array}{c}
\tilde A_1
\\
\end{array}
&  \displaystyle \{1,2,3\}& \quad\displaystyle s_{\gamma_1}\quad& n_{\gamma_1}&16
\\
\hline
\begin{array}{c}
\tilde A_1^{(2)}
\end{array}
&  \displaystyle  \{2,3\}& \quad\displaystyle s_{\beta_1}s_{\beta_2}\quad& n_{\beta_1}n_{\beta_2}&22\\
\hline
\begin{array}{c}
A_1\tilde A_1
\end{array}
&  \displaystyle \emptyset& \quad\displaystyle w_0\quad& n_{\beta_1}n_{\beta_2}n_{\beta_3}n_{\beta_4}&28\\
\hline
\end{array}
$$
\end{center}
\begin{center} Table \totable:  $F_4$, $\ch k=2$.
\end{center}
\medskip

\begin{proposition}\label{unip_F_4} Let $\O$ be a non-trivial unipotent conjugacy class in $F_4$. Then $\O$ is spherical if and only if it is listed in Table \ref{tavolauno}.
In each case $\O$ satisfies $(*)$. If $\ch k=2$, these are precisely the classes consisting of involutions.
\end{proposition}

We note that in $G_\C$ there is an involution $\s$ such that $C(\s)$ is of type $B_4$ and such that $\O_\s$ lies over $s_{\gamma_1}$. We also observe that if $p=2$, then $x_{\a_4}(1)\in A_1$,
$x_{\a_1}(1)\in \tilde A_1$, $x_{\a_1}(1)x_{\a_3}(1)\in A_1\tilde A_1$ and these exhaust the conjugacy classes of involutions with representative of the form $\prod_{i\in K}x_{\a_i}(1)$, $K\subseteq\Pi$. In particular $\tilde A_1^{(2)}$ has no representative of the form $\prod_{i\in K}x_{\a_i}(1)$.

\subsubsection{Type $G_2$.}
We put 
$
\b_1 =(3,2),\
\b_2=\a_1 
$,
$\gamma_1=(2,1)$ (the highest short root).

For groups of type $G_2$ we have to consider $p=2$, $3$. The p.o. set of unipotent conjugacy classes is described in the tables in \cite{spal}, Proposition II 10.4.

If $p=2$, the classification of unipotent conjugacy classes $\O$ for which $\dim \O\leq \dim B$ is the same as over $\C$ and each class consists of involutions. We may take
$$
x_{-\beta_1}(1)\in A_1\cap Bs_{\beta_1} B\cap U^-\quad,
$$
To deal with  $\tilde A_1$ we observe that $x=n_{\beta_1} n_{\b_2}$ is an involution in $Bw_0B$. Then $\dim \O_x\geq \ell(w_0)+\rk(1-w_0)=8$ by Theorem \ref{metodo}, so that $x\in\tilde A_1$ by Table \ref{tavolauno}.

\begin{center}
\vskip-20pt
$$
\begin{array}{|c||c|c|c|c|}
\hline
\O  &J& w(\O) & x\in \O\cap  Bw(\O)B&\dim \O \\ 
\hline
\hline 
\begin{array}{c}
A_1
\end{array} &\displaystyle \{1\} & s_{\beta_1}&n_{\b_1}&6\\
\hline
\begin{array}{c}
\tilde A_1
\end{array}
&  \emptyset& \quad\displaystyle w_0=s_{\b_1} s_{\b_2}\quad& n_{\b_1}n_{\b_2}&8
\\
\hline
\end{array}
$$
\end{center}
\begin{center} Table \totable:  $G_2$,
$p= 2$.
\end{center}

So now assume $p=3$. Here there are more conjugacy classes   $\O$ for which $\dim \O_u\leq \dim B$ (see Table \ref{tavolauno}), due to the presence of the graph automorphism of $G$. We may take the following representatives
$$
x_{-\beta_1}(1)\in A_1\cap Bs_{\beta_1} B\cap U^-\quad,
$$
$$
x_{-\gamma_1}(1)\in \tilde A_1\cap Bs_{\gamma_1}B\cap U^-\quad,
$$
To deal with $\tilde A_1^{(3)}$, we observe that since $A_1\leq \tilde A_1^{(3)}$ and $\tilde A_1\leq \tilde A_1^{(3)}$, we get $s_{\b_1}\leq w(\tilde A_1^{(3)})$ and $s_{\gamma_1}\leq w(\tilde A_1^{(3)})$, so that $w(\tilde A_1^{(3)})=w_0$, and we are done since $\dim \tilde A_1^{(3)}=8$.

\begin{center}
\vskip-20pt
$$
\begin{array}{|c||c|c|c|c|}
\hline
\O  &J& w(\O) &\dim \O \\ 
\hline
\hline 
\begin{array}{c}
A_1
\end{array} &\displaystyle \{1\} & s_{\b_1}&6\\
\hline

\begin{array}{c}
\tilde A_1
\\
\end{array}
&  \displaystyle \{2\}& \quad\displaystyle s_{\gamma_1}\quad&6
\\
\hline
\begin{array}{c}
\tilde A_1^{(3)}
\end{array}
&  \emptyset & \quad\displaystyle s_{\beta_1}s_{\beta_2}\quad& 8
\\
\hline
\end{array}
$$
\end{center}
\begin{center} Table \totable: $G_2$, $p=3$.
\end{center}
\medskip

\begin{proposition}\label{unip_G_2} Let $\O$ be a non-trivial unipotent conjugacy class in $G_2$. Then $\O$ is spherical if and only if it is listed in Table \ref{tavolauno}.
In each case $\O$ satisfies $(*)$. If $\ch k=2$, these are precisely the classes consisting of involutions.\cvd
\end{proposition}

Note that if $p=3$, then $x_{\a_2}(1)\in A_1$,
$x_{\a_1}(1)\in \tilde A_1$, while $\tilde A_1^{(3)}$ has no representative of the form $\prod_{i\in K}x_{\a_i}(1)$, $K\subseteq \Pi$.

This completes the classification of spherical unipotent conjugacy classes in bad characteristic.
In particular we have proved that

\begin{theorem}\label{finale} Let $\O$ be a non-trivial unipotent conjugacy class of a simple algebraic group in characteristic 2. Then $\O$ is spherical if and only if it consists of involutions.\cvd
\end{theorem}

This clearly holds for every connected reductive algebraic group in characteristic 2.

\begin{remark}
{\rm 
In each case there exists a unique maximal spherical conjugacy class $\O_{\rm max}$, and $w(\O_{\rm max})=w_0$. The union $\cal U^{\rm sph}$  of all spherical unipotent orbits is in the closure of $\O_{\rm max}$.
}
\end{remark}

\setcounter{equation}{0}
\section{Symmetric homogeneous spaces}\label{simmetrici}

In this section we shall prove that if $G$ is a connected reductive algebraic group over the algebraically close field $k$ of characteristic 2, then $H=C(\s)$ is a spherical subgroup of $G$ for every involutory automorphism $\s$ of $G$. This was proved by Vust in \cite{Vust} in characteristic zero (see also \cite{Matsuki} over $\C$). Then Springer extended the result to odd characteristic in \cite{Springerinv}. In \cite{Seitz}, Seitz gives an alternative proof of Springer's result. 

We shall use a generalization of Theorem \ref{metodo}. Here $G$ is any connected reductive algebraic group over $k$, any characteristic.

Let $\t$ be an automorphism of $G$ fixing $B$ and $T$, and consider $G:\<\t$. Assume $\t$ has order $r$. Then 
we have the Bruhat decomposition
$$
G:\<\t=\bigcup_{w\in W,\  i\in \Z_r}B\t^iwB
$$
Let $\O$ be a $G$-orbit in $G:\<\t$. Then there exists a unique $i\in \Z_r$ such that $\O\subseteq \bigcup_{w\in W}B\t^iwB$, and there is a unique $z=z(\O)$ such that 
$\O\cap B\t^izB$ is open dense in $\O$. In particular 
\begin{equation}
\overline{\O}=\overline {\O\cap
B\t^iz B}\subseteq \overline {B\t^iz B}= \t^i\,\overline {Bz B}.
\label{inclusion}
\end{equation}

It follows that if $y$ is an element of $\overline\O$ and $y\in B\t^i w B$, then
$w\leq z$ in the
Chevalley-Bruhat order of $W$. 
Let us observe that if $\O$ is spherical and
if $B.x$ is the dense $B$-orbit in $\O$, then $B.x\subseteq B\t^i zB$.

We still denote by $\t$ the automorphism of $E=X(T)\otimes \R$ induced by $\t$ (i.e. $\t(\chi)(t)=\chi(\t^{-1}t\t)$ for every $\chi\in X=X(T)$, $t\in T$). For every $w\in W$ we put
$$
T^{\t w}=\{t\in T ~|~ {w}^{-1}\t^{-1}t\t{w}=t\}
$$
We have $\dim T^{\t w}=n-\rk(1-\t w)$.

\begin{theorem}\label{metodoesteso}
Let $\s\in G:\<\t$, $\s=\t^i g$, for a certain $g\in G$, $i\in \Z_r$,  and let  $\O=G.\s$. 
Suppose that $\O$ contains an element $x\in B\t^iwB$, for a certain $w\in W$, where $U_w$ is $\t$-invariant.
Then 
$$
\dim B.x
\geq \ell(w)+\rk(1-\t^iw).
$$ 
In
particular $\dim {\O} \geq \ell(w)+\rk(1-\t^iw)$.
If, in addition, $\dim\O\leq\ell(w)+\rk (1-\t^iw)$ then $\O$ is spherical, 
$w=z(\O)$ and $B.x$ is the dense $B$-orbit in $\O$.\cvd
\end{theorem}
\begin{proof}
Without loss of generality,  we may assume $x=\t^i\dot w u$, for a certain representative $\dot w$ of $w$ in $N$ and $u\in U$. Let us estimate the dimension of the orbit $B_w. x$, where $B_w=TU_w$.

Let $vt\in C_{B_w}(x)$, with $v\in U_w$, $t\in T$.  Then
$$
\t^i\dot{w}uvt=vt\t^i\dot{w}u=\t^i\t^{-i}vt\t^i\dot{w}u=\t^i\t^{-i}v\t^i\t^{-i}t\t^i\dot{w}u=\t^i\,\t^{-i}v\t^i\,\dot{w}\,\dot{w}^{-1}\t^{-i}t\t^i\dot{w}u
$$
so that,
by the uniqueness of the decomposition, $v=1$
since $t^{-i}v\t^i\in U_w$. Moreover, from
$
ut=\dot{w}^{-1}\t^{-i}t\t^i\dot{w}u
$
it follows
$
t=\dot{w}^{-1}\t^{-i}t\t^i\dot{w}
$.
Therefore $C_{B_w}(x)\leq T^{\t^iw}$, thus
$\dim C_{B^w}(x)\leq \dim
T^{\t^iw}=n-\rk(1-\t^iw)$ and
$$
\dim B^w.x=\dim B^w-\dim
C_{B^w}(x)\geq
\ell(w)+n-n+\rk(1-\t^iw)=\ell(w)+\rk(1-\t^iw).
$$
If, in addition, $\ell(w)+\rk(1-\t^iw)\geq \dim\O$,
then $\dim\O=\ell(w)+\rk(1-\t^iw)$.
In particular $B.x$ is the dense $B$-orbit in $\O$.\cvd
\end{proof}

We observe that the condition $\t(U_w)=U_w$ is clearly satisfied if $w=s_{r_1}\cdots s_{r_k}$ where $r_1,\ldots,r_k$ are roots fixed by $\t$, or if $\{r_1,\ldots,r_k\}$ is a $\t$ invariant set of pairwise orthogonal roots.

In the remainder of this section, we assume that the characteristic of $k$ is 2.


We start with a general remark. Let $G=S\,G_1\cdots G_t$, where $S$ is the connected component of $Z(G)$, and $G_1,\ldots,G_t$ are the simple components of $G$. Let $\s$ be an involutorial automorphism of $G$. Then $\s$ fixes $S$, and induces a permutation $\r$ of the set $\{1,\ldots,t\}$. If $\r$ is non-trivial, then it is the product of disjoint cycles of length 2. Suppose one of these cycles is $(1,2)$. Then $\s$ induces an isomorphism $\f:G_1\to G_2$. Let $B_1=T_1U_1$ be a Borel subgroup of $G_1$, where $U_1$ is the unipotent radical of $B_1$, and $T_1$ a maximal torus. Let $V_1$ be the maximal unipotent subgroup of the Borel subgroup of $G_1$ opposite to $B_1$, and let $R=B_1V_1^\f\leq G_1G_2$. Then $C_{G_1G_2}(\s)\cap R$ is finite, so that 
$C_{G_1G_2}(\s)$ is a spherical subgroup of $G_1G_2$, since $\dim  C_{G_1G_2}(\s)+\dim B_1V_1^\f=\dim G_1G_2$. Of course $C_S(\s)$ is a spherical subgroup of $S$, so that to conclude it is enough to assume that $G$ is a (connected) simple algebraic group.

So now assume $G$ is a simple algebraic group. In the previous section we have already shown that $C(\s)$ is a spherical sugbroup of $G$ when $\s$ is an inner involution of $G$. We are therefore left to deal with outer involutions, which exists only in the following cases:
$A_\ell$, $\ell\geq 2$, $D_\ell$, $\ell\geq 4$ and $E_6$. 

To prove that the fixed point subgroup of any outer involution of $G$ is spherical, we shall use the classification of outer involutions of $G$ as in \cite{spal} and \cite{AS}. We fix the graph automorphism $\tau$ (of order 2) of $G$, and for each $G$-orbit $\O$ of outer involutions of $G$ we show that there exists $w=s_{\delta_1}\cdots s_{\delta_\ell}\in W$ such that  $\O\cap B\t wB$ is not empty, $\dim\O=\ell(w)+\rk(1-\t w)$, $\delta_i$'s are pairwise orthogonal positive roots and $\{\d_1,\ldots,\d_\ell\}$ is $\t$ invariant.
 By Theorem \ref{metodoesteso}, $\O$ is spherical (with $z(\O)=w$). We consider the various cases: if $\s=\tau g\in G:\<\tau$, $C(\s)$ stands for $C_G(\s)$. In each case we use the notation introduced in section \ref{classificazione}.

\medskip

\subsection{Type $A_n$, $n=2m$, $m\geq 1$.}

We take $G=SL(2m+1)$. In this case there is only one (class of) outer involution $\t$, the graph automorphism of $SL(2m+1)$, and $C(\t)=SO(2m+1)$. Then 
$$
\dim SL(2m+1)/SO(2m+1)=2m^2+3m
$$
which is the dimension of a Borel subgroup of $SL(2m+1)$.
We may take 
$$
x=\tau n_{\b_1}\cdots n_{\b_m}\in \tau w_0B.
$$
Then $x^2=1$ since $\tau (\b_k) ={\b_k}$ for each $k$, so that $x$ lies in $\O_\tau$.
 Since $\tau w_0=-1$, we get 
 $$
 \ell(w_0)+\rk(\t w_0-1)=\dim B
 $$ and we are done. 
 
 Hence
 \begin{center}
\vskip-20pt
$$
\begin{array}{|c||c|c|}
\hline
\O  &w(\O)&x\in \O\cap \tau Bw(\O)B \\ 
\hline
\hline
\begin{array}{c}
\tau
\end{array}
&\quad \displaystyle w_0=s_{\b_1}\cdots s_{\b_m}\quad&\quad\displaystyle \tau n_{\b_1}\cdots n_{\b_m}\quad \\
\hline
\end{array}
$$
\end{center}
\begin{center} Table \totable:  Outer involutions in $SL(2m+1)$,
$m\geq 1$.
\end{center}

\subsection{Type $A_n$, $n=2m-1$, $m\geq 2$.}

We take $G=SL(2m)$. In this case there are two (classes of)  outer involutions: $\t$ and $\tau x_{\beta_1}(1)$, with $C(\t)=Sp(2m)$, $C(\tau x_{\beta_1}(1))=C_{Sp(2m)}(x_{\beta_1}(1))$. We have
$$
\dim SL(2m)/Sp(2m)=2m^2-m-1
$$
We put $J=\{1,3,\ldots, n\}$, $w=w_0\wJ$. We have 
$$
\ell(w)=\ell(w_0)-\ell(\wJ)=m(2m-1)-m=2m^2-2m
$$
and
$$
\rk(\tau w-1)=\rk(\wJ+1)=m-1
$$
since $\tau w_0=-1$ and the (simple) roots in $J$ are pairwise orthogonal. Hence
$$
\ell(w)+\rk(\tau w-1)=2m^2-m-1=\dim SL(2m)/Sp(2m)
$$
We are left to exhibit a conjugate $x\in \tau BwB$ of $\tau$. For this purpose we distinguish 2 cases. 

Assume $m$ is even, $m=2r$. Then there are precisely $m$ positive roots $\gamma_1,\ldots,\gamma_m$ for which $w(\gamma)=-\gamma$, namely
$$
\g_{2i-1}=e_{2i-1}-e_{2m-2i+1}\quad, \quad
\g_{2i}=e_{2i}-e_{2m+2-2i}
$$
for $i=1,\ldots,r$ and
$$
w=s_{\g_1}\cdots s_{\g_m}
$$
We also note that $\tau$ exchanges $\g_{2i-1}$ and $\g_{2i}$ for each $i=1,\ldots,r$.
In this case we take 
$$
x=g\tau g^{-1}
$$
where $g$ is the involution $g=x_{-\g_1}(1)x_{-\g_3}(1)\cdots x_{-\g_{m-1}}(1)$. Then
$$
x=\tau x_{-\g_1}(1)x_{-\g_2}(1)x_{-\g_3}(1)\cdots x_{-\g_{m-1}}(1)x_{-\g_m}(1)\in \tau BwB
$$
and we are done.

If $m$ is odd, $m-1=2r$, then there are precisely $m-1$ positive roots $\gamma_1,\ldots,\gamma_{m-1}$ for which $w(\gamma)=-\gamma$, namely
$$
\g_{2i-1}=e_{2i-1}-e_{2m-2i+1}\quad, \quad
\g_{2i}=e_{2i}-e_{2m+2-2i}
$$
for $i=1,\ldots,r$ and
$$
w=s_{\g_1}\cdots s_{\g_{m-1}}
$$
We also note that $\tau$ exchanges $\g_{2i-1}$ and $\g_{2i}$ for each $i=1,\ldots,r$.
In this case we take 
$$
x=g\tau g^{-1}
$$
where $g$ is the involution $g=x_{-\g_1}(1)x_{-\g_3}(1)\cdots x_{-\g_{m-2}}(1)$. Then
$$
x=\tau x_{-\g_1}(1)x_{-\g_2}(1)x_{-\g_3}(1)\cdots x_{-\g_{m-2}}(1)x_{-\g_{m-1}}(1)\in \tau BwB
$$
and we are done.

We finally deal with $\tau x_{\beta_1}(1)$, $H=C(\tau x_{\beta_1}(1))=C_{Sp(2m)}(x_{\beta_1}(1))$. We have
$$
\dim SL(2m)/H=\dim SL(2m)/Sp(2m)+\dim \O'=2m^2-m-1+2m=2m^2+m-1
$$
where $\O'$ is the $Sp(2m)$-orbit of $x_{\beta_1}(1)$, which has dimension $2m$ (note that $\dim SL(2m)/H=\dim SL(2m)/SO(2m)$, and $SO(2m)$ is the centralizer of an outer involution of $SL(2m)$ if the characteristic is not 2). Therefore $\dim SL(2m)/H$ is the dimension of a Borel subgroup of $SL(2m)$.
 As in the case  when $n$ is even, we take $w=w_0$, 
$$
x=\tau n_{\b_1}\cdots n_{\b_m}\in \tau w_0B.
$$
 Then $x^2=1$ since $\tau(\b_k) =\b_k$ for each $k$, and for dimensional reasons, $x$ is conjugate to $\tau x_\beta(1)$.
 
 Hence  \begin{center}
\vskip-20pt
$$
\begin{array}{|c||c|c|c|}
\hline
\O  & w(\O) & x\in \O\cap \tau Bw(\O)B \\ 
\hline
\hline 
\begin{array}{c}
\tau
\end{array} & s_{\g_1}\cdots s_{\g_m} & \tau x_{-\g_1}(1)\cdots x_{-\g_m}(1)\\
\hline
\begin{array}{c}
\tau x_{\b_1}(1)
\end{array}
&   \quad\displaystyle w_0\quad&\tau n_0
\\
\hline
\end{array}
$$
\end{center}
\begin{center} Table \totable:  Outer involutions in $SL(2m)$,
$m\geq 2$, $m$ even.
\end{center}

and

 \begin{center}
\vskip-20pt
$$
\begin{array}{|c||c|c|c|}
\hline
\O  & w(\O) & x\in \O\cap \tau Bw(\O)B \\ 
\hline
\hline 
\begin{array}{c}
\tau
\end{array}  & s_{\g_1}\cdots s_{\g_{m-1}} & \tau x_{-\g_1}(1)\cdots x_{-\g_{m-1}}(1)\\
\hline
\begin{array}{c}
\tau x_{\b_1}(1)
\end{array}
& \quad\displaystyle w_0\quad&\tau n_0
\\
\hline
\end{array}
$$
\end{center}
\begin{center} Table \totable:  Outer involutions in $SL(2m)$,
$m\geq 3$, $m$ odd.
\end{center}

\medskip
\subsection{Type $D_n$, $n\geq 4$.}

To deal with $G$ of type $D_n$ we shall, as usual, consider $G=SO(2n)$. Then the outer involutions of $G$ are obtained by conjugation with involutions of $O(2n)$. Note that if  $n=4$,  and if $G$ is adjoint or simply-connected, then there are other outer involutions in $\Aut G$: however, they are conjugate in $\Aut G$.

Let $\tau$ be the involution of $O(2n)$ inducing the graph automorphism of $SO(2n)$, i.e. the graph automorphism acting trivially on $\<{X_{\pm\a_i}\mid i\in \{1,\ldots,n-2\}}$ and such that $x_{\a_{n-1}}(\xi)\leftrightarrow x_{\a_{n}}(\xi)$, $x_{-\a_{n-1}}(\xi)\leftrightarrow x_{-\a_{n}}(\xi)$ for $\xi\in k$.

The classes of involutions in $O(2n)\setminus SO(2n)$ correspond to partitions
$2^{k}\oplus 1^{2n-2k}$ for $k=1,\ldots,n$, odd $k$, with $\t$ corresponding to $2\oplus 1^{2n-2}$. Let $\O_k$ be the class corresponding to $2^{k}\oplus 1^{2n-2k}$.
From \cite{spal}, 2.9 b) we get
$$
\dim \O_k=\dim \O_{k, Sp(2n,\C)}-2n+\l_1^\ast
$$
where
$\l_1^\ast=c_1+c_2=(2n-2k)+k=2n-k$, 
hence
$$
\dim \O_k=k(2n-k+1)-2n+2n-k=k(2n-k)
$$
for $k=1,\ldots,n$, odd $k$.

Let
$$
\mu_1=e_1-e_{n}\quad,\quad
\nu_1=e_1+e_{n}\quad,\quad w=s_{\mu_1}s_{\nu_1}\quad.
$$
Then
$$
\ell(w)+\rk(1-\tau w)=2(n-1)+1=2n-1=\dim \O_1
$$
and $\t(\mu_1)=\nu_1$. We have
$$
n_{\nu_1}\tau n_{\nu_1}=\tau n_{\mu_1}n_{\nu_1}
$$
so that 
$$
\tau n_{\mu_1}n_{\nu_1}\in \O_1\cap wB
$$
To deal with the remaining classes, we put $m=\left[\frac{n}2\right]$ and
$$
\mu_i=e_{2i-2}-e_{2i-1}\quad,\quad
\nu_i=e_{2i-2}+e_{2i-1}\quad,\quad w_i=s_{\mu_1}s_{\nu_1}\cdots s_{\mu_i}s_{\nu_i}
$$
for $i=2,\ldots,m$.

Arguing as above, we can prove that
$$
\ell(w)+\rk(1-\tau w)=\dim \O_{2i-1}
$$
and
$$
\tau n_{\mu_1}n_{\nu_1}\cdots n_{\mu_i}n_{\nu_i}\in \O_{2i-1}\cap w_iB
$$
for  $i=2,\ldots,m$. In fact
it is enough to count the number of Jordan blocks of length 2 in $\tau n_{\mu_1}n_{\nu_1}\cdots n_{\mu_i}n_{\nu_i}$: in $\tau n_{\mu_1}n_{\nu_1}$ there is 1, and in $n_{\mu_i}n_{\nu_i}$ there are 2 for each $i=2,\ldots,m$.

If $n$ is even, then there are $\frac12n$ conjugacy classes of outer involutions and we are done. In particular the maximal one is $2^{n-1}\oplus 1^{2}$ and corresponds to $w=s_{\mu_1}s_{\nu_1}\cdots s_{\mu_m}s_{\nu_m}=w_0=-1$.

If $n$ is odd, then there are $\frac12(n+1)$ conjugacy classes of outer involutions: the maximal one is $\O_n=2^{n}$ which is the only one not in the previous list.  We have 
$$
\dim \O_n=n^2
$$
Let $n_0$ be any representative in $N$ of $w_0$ with $n_0$ of order 2 and commuting with $\tau$. Then $x=\tau n_0$ is an involution in $\tau w_0B$. Since 
$$
\ell(w_0)+\rk(1-\tau w_0)=n^2-n+n=n^2
$$
by Theorem \ref{metodoesteso}, we have $\dim B.x\geq n^2$. But $x$ is an involution in $O(2n)\setminus SO(2n)$, so that $\dim G.x\leq \dim \O_n=n^2$. Therefore  $x$ lies in $\O_n\cap \tau w_0B$ and we are done.

\begin{center}
\vskip-20pt
$$
\begin{array}{|c||c|c|c|c|}
\hline
\O  & w(\O) & x\in \O\cap \tau Bw(\O)B&\dim \O \\ 
\hline
\hline 
\begin{array}{c}
2^{2i-1}\oplus 1^{2n-4i+2}
\\
i=1,\ldots,m
\end{array}
& \quad\displaystyle s_{\mu_1}s_{\nu_1}\cdots s_{\mu_i}s_{\nu_i}\quad&\tau n_{\mu_1}n_{\nu_1}\cdots n_{\mu_i}n_{\nu_i}&(2i-1)(2n-2i+1)
\\
\hline
\end{array}
$$
\end{center}
\begin{center} Table \totable:  Outer involutions in $D_n$,
$n\geq 4$, $n=2m$.
\end{center}

\begin{center}
\vskip-20pt
$$
\begin{array}{|c||c|c|c|c|}
\hline
\O  & w(\O) & x\in \O\cap \tau Bw(\O)B&\dim \O \\ 
\hline
\hline 
\begin{array}{c}
2^{2i-1}\oplus 1^{2n-4i+2}
\\
i=1,\ldots,m
\end{array}
&   \quad\displaystyle s_{\mu_1}s_{\nu_1}\cdots s_{\mu_i}s_{\nu_i}\quad& \tau n_{\mu_1}n_{\nu_1}\cdots n_{\mu_i}n_{\nu_i}&(2i-1)(2n-2i+1)
\\
\hline
\begin{array}{c}
2^{n}
\end{array}
& \quad\displaystyle w_0\quad& \tau n_{0} &n^2\\
\hline
\end{array}
$$
\end{center}
\begin{center} Table \totable:  Outer involutions in $D_n$,
$n\geq 4$, $n=2m+1$.
\end{center}

 \medskip

\subsection{Type $E_6$.}\label{nane}

There are two (classes of) outer involutions: $\t$ and $\tau x_{\beta_1}(1)$, where $\tau$ is the graph automorphism of $G$.
We recall from \S \ref{nane} that
$$
\begin{array}{cclll}
\b_1 = (1,2,2,3,2,1),&
\b_2 =(1,0,1,1,1,1)\\
\b_3 =(0,0,1,1,1,0),&
\b_4 = (0,0,0,1,0,0)
\end{array}
$$
Note that each $\b_i$ is fixed by $\t$.

Let us start with $\tau x_{\beta_1}(1)$. 
We have $K=C(\tau x_{\beta_1}(1))\cong C_{F_4} (x_{\beta_1}(1))$, $\dim K=36$. Let $x=\tau n_{\b_1}n_{\b_2}n_{\b_3}n_{\b_4}
$. Since $x
$ is an involution
 in $\tau  w_0B$, with
$\ell(w_0)+\rk(\tau w_0-1)=36+6=\dim E_6/K$, it follows that
$
\tau x_{\beta_1}(1)\sim x$.

To deal with $\tau$, we put $\delta_1=(1, 1, 2, 2, 1, 1)$, $\delta_2=(1, 1, 1, 2, 2, 1)$. We have 
$$
\dim E_6/F_4=26
$$
Note that $\tau(\delta_1)=\delta_2$ and 
$$
\ell(s_{\d_1}s_{\d_2})+\rk(\tau  s_{\d_1}s_{\d_2}-1)=24+2=26
$$
In fact here $J=\{2,3,4,5\}$, $w=s_{\d_1}s_{\d_2}=w_0\wJ$.

We show that $\tau\sim \tau n_{\d_1}n_{\d_2}$.
Let $g=x_{-\d_1}(1)$. Then
$$
g\tau g^{-1}=\tau x_{-\d_1}(1)x_{-\d_2}(1)\in \tau BwB
$$
Moreover, since $[\tau, x_{-\d_1}(1)x_{-\d_2}(1)]=1$, we get
$$
x_{\d_1}(1)x_{\d_2}(1)\tau x_{-\d_1}(1)x_{-\d_2}(1)x_{\d_1}(1)x_{\d_2}(1)=\tau n_{\d_1}n_{\d_2}
$$
and we are done.
 \medskip
 We summarize in

 \begin{center}
\vskip-20pt
$$
\begin{array}{|c||c|c|c|}
\hline
\O  & w(\O) & x\in \O\cap \tau Bw(\O)B \\ 
\hline
\hline 
\begin{array}{c}
\tau
\end{array} & s_{\d_1}s_{\d_2} & \tau n_{\d_1}n_{\d_2}\\
\hline
\begin{array}{c}
\tau x_{\b_1}(1)
\end{array}
&   \quad\displaystyle w_0\quad&\tau n_{\b_1}n_{\b_2}n_{\b_3}n_{\b_4}
\\
\hline
\end{array}
$$
\end{center}
\begin{center} Table \totable:  Outer involutions in $E_6$.
\end{center}

\medskip
This completes the list of outer involutions of simple algebraic groups in characteristic 2.
We have proved that

\begin{theorem}\label{finalesimmetrico} 
Let $G$ be a reductive connected algebraic group in characteristic 2, and let $\s$ be any involutory automorphism of $G$. Then  the fixed point subgroup $H$ of $\s$ is a spherical subgroup of $G$.\cvd
\end{theorem}

We conclude with another application of Theorem \ref{metodoesteso}.

\medskip
\subsection{Type $G_2$ in $D_4$.}

We show briefly how one can prove that the subgroup of type $G_2$ in $D_4$ is spherical (in all characteristics). Let us assume $G$ of type $D_4$. Without loss of generality, we may assume $G$ is adjoint. Hence if we denote by $\f$ the graph automorphism of $G$ fixing $\a_2$ and mapping $\a_1\mapsto \a_3\mapsto \a_4\mapsto \a_1$, then the fixed point subgroup $K$ of $\f$ is of type $G_2$. Let
$\d_1=\a_1+\a_2+\a_3$,
$\d_2=\a_1+\a_2+\a_4$,
$\d_3=\a_2+\a_3+\a_4$, and let $w=w_0 s_2=s_{\d_1}s_{\d_2}s_{\d_3}$. We have 
$$
\ell(w)+\rk(1-\f w)=14=\dim D_4/G_2\quad.
$$
It remains to show that a $G$-conjugate of $\f$ lies in  $\f B w_0s_2B\subseteq G:\<\f$.

Let $g=x_{-\d_1}(\xi_1)x_{-\d_2}(\xi_2)x_{-\d_3}(\xi_3)$. Then
$$
g\f g^{-1}=\f x_{-\d_1}(-\xi_1-\xi_3)x_{-\d_2}(-\xi_2-\xi_1)x_{-\d_3}(-\xi_3+\xi_2)
$$
If we choose $\xi_1$, $\xi_2$, $\xi_3$ such that $\xi_1+\xi_3$, $\xi_2+\xi_1$ and $-\xi_3+\xi_2
$ are non-zero, then 
$g\f g^{-1}\in \f B w_0s_2B$ and we are done.

\begin{remark}
{\rm 
If the characteristic is zero, we may apply the arguments in \cite{Cos}. Since $T^{\f w}=H_{\a_2}$ is connected, it follows that, in the simply-connected case, the monoid $\l(D_4/G_2)$ of $B$-weights in $k[D_4/G_2]$ is generated by $\omega_1, \omega_3,\omega_4$, since $G_2$ is connected and it has no non-trivial characters, so that the monoid $\l(D_4/G_2)$ is free (and it contains $(1-\f w)P^+$ which is the monoid generated by $\omega_1+\omega_3$, $\omega_1+\omega_4$, $\omega_3+\omega_4$, $\omega_1+\omega_3+\omega_4$) (see also \cite{Kramer}).
}
\end{remark}


\begin{thebibliography}{1}

\bibitem{Ale}
{\sc A.\ V.\ Alekseevskii,}
\newblock{\em Component groups of centralizers of unipotent elements in semisimple algebraic groups,}
\newblock Trudy Tbiliss. Mat. Inst. Razmadze Akad. Nauk. Gruzin SSR 62, 5--27 (1979),
\newblock{Lie groups and invariant theory,}
\newblock Amer. Math. Soc. Transl. Ser. 2, 213, 
\newblock Amer. Math. Soc. Providence, RI (2005).


\bibitem{AS}{\sc M.\ Aschbacher, G.M.\ Seitz}
\newblock {\em Involutions in Chevalley groups over fields of even order,}
\newblock Nagoya Math. J., 63, 1--91 (1976).


\bibitem{bourbaki}{\sc N.\ Bourbaki,}
\newblock{\em \'El\'ements de Math\'ematique. Groupes et Alg\`ebres de Lie, Chapitres 4,5, et 6,}
\newblock Masson, Paris (1981). 


\bibitem{Brundan}
{\sc J.\ Brundan,}
\newblock{\em Dense orbits and double cosets,}
\newblock Algebraic groups and their representations (Cambridge, 1997),  259--274, NATO Adv. Sci. Inst. Ser. C Math. Phys. Sci.,
517, Kluwer Acad. Publ., Dordrecht, (1998).

\bibitem{CCC}{\sc N.\ Cantarini, G.\ Carnovale, M.\ Costantini,}
\newblock {\em Spherical orbits and representations of $\Ue$,}
\newblock Transformation\  Groups, 10, No.\ 1, 29--62  (2005).

\bibitem{Car}{\sc G.\ Carnovale,}
\newblock {\em Spherical conjugacy classes and involutions in the Weyl group,}
\newblock Math. Z. 260(1), 1--23 (2008)

\bibitem{Car2}{\sc G.\ Carnovale,}
\newblock {\em Spherical conjugacy classes and Bruhat decomposition,}
\newblock preprint ArXiv:0808.1818v2 (to appear in Ann. Inst. Fourier, Grenoble).

\bibitem{Giovanna-good}{\sc G.\ Carnovale,}
\newblock {\em A classification of spherical conjugacy classes 
in good characteristic,}
\newblock preprint ArXiv:0811.2641v2 (to appear in Pacific J. Math.).

\bibitem{Cuno}
{\sc M.\ Costantini}
\newblock {\em On the lattice automorphisms of certain simple
algebraic groups,}
\newblock Rend.\ Sem.\ Mat.\ Univ.\ Padova 90,
141--157 (1993). 

\bibitem{Cdue}
{\sc M.\ Costantini}
\newblock {\em The lattice automorphisms of simple
algebraic groups over ${\overline{\mathbb F}}_2$,}
\newblock Manuscripta Math. 91, 1--16 (1996).

\bibitem{Cos}{\sc M.\ Costantini,}
\newblock {\em On the coordinate ring of spherical conjugacy classes,}
\newblock Math. Z. on line (2009)

\bibitem{Carter1}
{\sc R.\ W.\ Carter,}
\newblock {\em Simple Groups of Lie Type,}
\newblock John Wiley (1989).

\bibitem{Carter2}
{\sc R.\ W.\ Carter,}
\newblock {\em Finite Groups of Lie Type,} 
\newblock John Wiley  (1985).

\bibitem{intersezioni}{\sc K.Y.\ Chan, J.-H.\ Lu, S.K.M.\ To,}
\newblock {\em On intersections of conjugacy classes and Bruhat cells,}
\newblock preprint ArXiv:0906.2254v1.


\bibitem{FR}{\sc R.\ Fowler, G.\ R\"ohrle}
\newblock{\em Spherical nilpotent orbits in positive characteristic,}
\newblock Pacific J. Math. 237, 241--286 (2008).
 

\bibitem{Hesselink}{\sc W.H.\ Hesselink,}
\newblock{\em Nilpotency in classical groups over a filed of characteristic 2,}
\newblock Math. Z. 54, 165--181 (1979).


\bibitem{Hu2}
{\sc J.E.\ Humphreys,}
\newblock {\em Linear Algebraic Groups,}
\newblock Springer-Verlag, New York (1995).


\bibitem{Iwa}
{\sc N.\ Iwahori,}
\newblock{\em Centralizers of involutions in finite Chevalley groups,}
\newblock In: ``Seminar on algebraic groups and related finite groups''.
LNM 131, 267--295, Springer-Verlag, Berlin Heidelberg New York (1970).


\bibitem{Knop}
{\sc F. \ Knop,}
\newblock{\em On the set of orbits for a Borel subgroup,}
\newblock Comment. Math. Helv.\ 70, 285--309 (1995).


\bibitem{Kramer}
{\sc M.\ Kr\"amer,}
{\it Sph\"arische Untergruppen in kompakten zusammenhangenden
Liegruppen,}
Compositio Math.\ {\bf 38} (1979),  129--153.

\bibitem{LLS}{\sc R.\ Lawther, M.W.\ Liebeck, G.M.\ Seitz}
\newblock {\em Fixed point spaces in actions of exceptional algebraic groups,}
\newblock Pacific J. Math. 205, 339--391 (2002).


\bibitem{Lus}{\sc G.\ Lusztig,}
\newblock{\em Remarks on Springer's representations,}
\newblock preprint ArXiv:0811.0370v1 (3 Nov 2008).

\bibitem{Matsuki}
{\sc T.\ Matsuki,}
\newblock{\em The orbits of affine symmetric spaces under the action of minimal parabolic subgroups,}
\newblock J. Math. Soc. Japan\ 31, 331--357 (1979).


\bibitem{pany3}{\sc D.\ Panyushev,}
\newblock {\em Complexity and rank of homogeneous spaces,}
\newblock Geom.\ Dedicata, 34, 249--269 (1990).

\bibitem{pany}{\sc D.\ Panyushev,}
\newblock {\em Complexity and nilpotent orbits,}
\newblock Manuscripta Math. 83, 223--237 (1994).


\bibitem{pany2}{\sc D.\ Panyushev,}
\newblock {\em On spherical nilpotent orbits and beyond,}
\newblock Ann.\ Inst.\ Fourier, Grenoble 49(5), 1453--1476 (1999).

\bibitem{pany5}{\sc D.\ Panyushev,}
\newblock {\em Some amazing properties of spherical nilpotent orbits,}
\newblock Math. Z., 245, 557--580 (2003).


\bibitem{Seitz}
{\sc G.M..\ Seitz,}
\newblock{\em Double cosets in algebraic groups,}
\newblock Algebraic groups and their representations (Cambridge, 1997),  214--257, NATO Adv. Sci. Inst. Ser. C Math. Phys. Sci.,
517, Kluwer Acad. Publ., Dordrecht, (1998).



\bibitem{spal}
{\sc J.N. Spaltenstein,}
\newblock{\em Classes unipotentes et sous-groupes de Borel,} 
\newblock LNM 946, Springer-Verlag, Berlin Heidelberg New York  (1982).


\bibitem{Springerinv}
{\sc T.A. Springer,}
\newblock{\em Some results on algebraic groups with involutions,}
\newblock Algebraic groups and related topics (Kyoto/Nagoya, 1983),  525--543, Adv. Stud. 
Pure Math. 6, North-Holland, Amsterdam, (1985).

\bibitem{springer}
{\sc T.A. Springer,}
\newblock{\em Linear Algebraic Groups,} Second Edition,
\newblock Progress in Mathematics 9, Birkh{\"a}user (1998).


\bibitem{yale}
\newblock{\sc R.\ Steinberg,}
\newblock{\em Lectures on Chevalley groups,}
\newblock{Yale University} (1967).

\bibitem{Vust}{\sc T. Vust,}
\newblock{\em Op\'eration de groupes r\'eductifs dans un type de c\^ones presques homog\`enes,}
\newblock Bull. Soc. Math. France \ 102, 317--333 (1974).

\bibitem{Wall}{\sc G.E. Wall,}
\newblock{\em On the conjugacy classes in the unitary, symplectic and orthogonal groups,}
\newblock J. Austral. Math. Soc. \ 3, 62--89 (1963).


\end{thebibliography}
\end{document}